\documentclass[a4paper,11pt]{amsart}

\usepackage[textsize=footnotesize,color=blue!40, bordercolor=white] 
{todonotes}

 \usepackage
[a4paper, margin=1.2in]{geometry}

\usepackage{hyperref} 

\usepackage{xcolor}
\definecolor{blue}{rgb}{0,0,1}
\hypersetup{
	unicode=true,
	colorlinks=true,
	citecolor=blue,
	linkcolor=blue,
	anchorcolor=blue
}

\usepackage[latin1]{inputenc}
\usepackage[T1]{fontenc}
\usepackage[english]{babel}

\usepackage{mathrsfs}
\usepackage{amscd}
\usepackage{amsfonts}
\usepackage{amsmath}
\usepackage{amssymb}
\usepackage{amstext}
\usepackage{amsthm}
\usepackage{amsbsy}

\usepackage{xspace}
\usepackage[all]{xy}
\usepackage{graphicx}
\usepackage{url}
\usepackage{latexsym}

\makeatletter
\newcommand*{\rom}[1]{\expandafter\@slowromancap\romannumeral {\sharp}1@}
\makeatother

\theoremstyle{definition}

\newcommand{\RR}{\mathcal{R}} 

\newcommand{\ITTM}{$\mathrm{ITTM}$} 
\newcommand{\ML}{$\mathrm{ITTM_{\mathrm{ML}}}$} 
\newcommand{\om}{\omega_1^{\mathrm{ck}}}

\newtheorem{fact}{fact}[section]

\newtheorem{thm}[fact]{Theorem}
\newtheorem{lemma}[fact]{Lemma}

\newtheorem{definition}[fact]{Definition}
\newtheorem{defini}[fact]{Definition}
\newtheorem{remark}[fact]{Remark}

\newtheorem{question}[fact]{Question}

\newtheorem{claim}[fact]{Claim}
\newtheorem*{claim*}{Claim}
\newtheorem*{subclaim*}{Subclaim}

\usepackage{amsmath,amssymb,amscd,amsthm,amsfonts,amstext,amsbsy,mathrsfs,hyperref,upgreek,mathtools,stmaryrd,enumitem}

\newenvironment{enumerate-(a)}{\begin{enumerate}[label={\upshape (\alph*)}, leftmargin=2pc]}{\end{enumerate}}

\newenvironment{enumerate-(a)-r}{\begin{enumerate}[label={\upshape (\alph*)}, leftmargin=2pc,resume]}{\end{enumerate}}

\newenvironment{enumerate-(A)}{\begin{enumerate}[label={\upshape (\Alph*)}, leftmargin=2pc]}{\end{enumerate}}

\newenvironment{enumerate-(A)-r}{\begin{enumerate}[label={\upshape (\Alph*)}, leftmargin=2pc,resume]}{\end{enumerate}}

\newenvironment{enumerate-(i)}{\begin{enumerate}[label={\upshape (\roman*)}, leftmargin=2pc]}{\end{enumerate}}

\newenvironment{enumerate-(i)-r}{\begin{enumerate}[label={\upshape (\roman*)}, leftmargin=2pc,resume]}{\end{enumerate}}

\newenvironment{enumerate-(I)}{\begin{enumerate}[label={\upshape (\Roman*)}, leftmargin=2pc]}{\end{enumerate}}

\newenvironment{enumerate-(I)-r}{\begin{enumerate}[label={\upshape (\Roman*)}, leftmargin=2pc,resume]}{\end{enumerate}}

\newenvironment{enumerate-(1)}{\begin{enumerate}[label={\upshape (\arabic*)}, leftmargin=2pc]}{\end{enumerate}}

\newenvironment{enumerate-(1)-r}{\begin{enumerate}[label={\upshape (\arabic*)}, leftmargin=2pc,resume]}{\end{enumerate}}

\title[Randomness via infinite computation]{Randomness via infinite computation and effective descriptive set theory} 

\author{Merlin Carl and Philipp Schlicht}

\begin{document}

\maketitle  

\begin{abstract} 
We study randomness beyond $\Pi^1_1$-randomness and its Martin-L\"of type variant, introduced in \cite{MR2340241} and further studied in \cite{Continuous-higher-randomness}. 

The class given by the infinite time Turing machines (\ITTM s), introduced by Hamkins and Kidder, is strictly between $\Pi^1_1$ and $\Sigma^1_2$. 
We prove that the natural randomness notions associated to this class have several desirable properties resembling those of the classical random notions such as Martin-L\"of randomness, and randomness notions defined via effective descriptive set theory such as $\Pi^1_1$-randomness. 
For instance, mutual randoms do not share 
information and can be characterized as in van Lambalgen's theorem. 
We also obtain some differences to the hyperarithmetic setting. 
Already at the level of $\Sigma^1_2$, some properties of randomness notions are independent \cite{Infinite-computations}. 

Towards the results about randomness, we prove the following analogue  to a theorem of Sacks. 
If a real is infinite time Turing computable relative to all reals in some given set of reals with positive Lebesgue measure, then it is already infinite time Turing computable. 
As a technical tool, we prove facts of independent interest about random forcing over admissible sets and increasing unions of admissible sets. 
These results are also useful for more efficient proofs of some classical results about hyperarithmetic sets. 
\end{abstract}

\setcounter{tocdepth}{2} 
\tableofcontents

\section{Introduction} 


Algorithmic randomness studies formal notions that express the intuitive concept of an \emph{arbitrary} or \emph{random} infinite bit sequence 
with respect to Turing programs. 
The most prominent such notion is \emph{Martin-L\"of randomness} ($\mathrm{ML}$). 
A real number, i.e. a sequence of length the natural numbers with values $0$ and $1$, 
is $\mathrm{ML}$-random if and only if it is not contained
in a set of Lebesgue measure $0$ that can be effectively approximated by a Turing machine in a precise sense. 
We refer the reader to comprehensive treatments of this topic in \cite{MR2732288, MR2548883}. 

Martin-L\"of already suggested that the classical notions of randomness are too weak. Moreover, Turing computability is relatively weak  
in comparison with notions in descriptive set theory. 
Therefore higher notions of randomness have been considered, for instance, computably enumerable sets are replaced with $\Pi^1_1$ sets (see \cite{MR2340241, Continuous-higher-randomness}). 
These notions were recently studied in \cite{Continuous-higher-randomness}, and in particular the authors defined a continuous relativization which allowed them to prove a variant of van Lambalgen's theorem for $\Pi^1_1$-ML-random. We will use this or the Martin-L\"of variant of ITTM-random reals in Section \ref{subsection ML}. 

There are various desirable properties for a notion of randomness, which many of the formal notions possess, and which can serve as criteria for the evaluation of such a notion. 
For instance, different approaches to the notion of randomness, 
such as not having effective rare properties, being incompressible or being unpredictable are often equivalent. 
Van Lambalgen's theorem states that each half of a random sequence is random with respect to the other half. 
Moreover, there is often a universal test. 
For instance $\mathrm{ML}$-randomness and its $\Pi^1_1$-variant (see  \cite{MR2340241} and \cite{Continuous-higher-randomness} for the relativization) satisfy these conditions. 
Some types of random reals are not informative and real numbers that are mutually random do not share any nontrivial information. This does not hold for $\mathrm{ML}$-randomness and its variant at the level of $\Pi^1_1$, but it does hold for $\Pi^1_1$-randomness and the notion of \ITTM-randomness studied in this paper. 

Higher randomness studies properties of classical randomness notions for higher variants. Various results can be extended to higher randomness notions, assuming sufficiently large cardinals (see e.g. \cite{Yu}). 
However, already at the level of $\Sigma^1_2$, many properties of randomness notions are independent \cite{Infinite-computations}. 
Therefore we consider classes strictly between $\Pi^1_1$ and $\Sigma^1_2$. 

The infinite time Turing machines introduced by Hamkins and Kidder (see \cite{MR1771072}) combine the appeal of machine models with considerable strength. 
The notions decidable, semi-decidable, computable, writable etc. will refer to these machines. 
The strength of these machines is strictly above $\Pi^1_1$ and therefore, this motivates the consideration of notions of randomness based on \ITTM s. 
 This project was started in \cite{Infinite-computations} and continued in \cite{Randomness-and-degree-theory-for-ITRMs, Infinite-time-recognizability-from-random-oracles}. 

We consider the following notions of randomness for a real. 
\begin{itemize} 
\item 
\ITTM-random: avoids every semidecidable null set, 
\item 
\ITTM-decidable random: avoids every decidable null set, 
\item 
\ML-random:  like $\mathrm{ML}$-randomness, but via ITTMs instead of Turing machines.
\end{itemize} 

With respect to the above criteria, they perform differently. 
As we show below, all notions satisfy van Lambalgen's theorem. 
We will see that there is a universal test for \ITTM-randomness and \ML-randomness, but not for ITTM-decidable randomness, and we will relate these notions to randomness over initial segments of the constructible hierarchy. 
A new pohenomenon for ITTMs compared to the hyperarithmetic setting is the existence of \emph{lost melodies}, i.e. non-computable recognizable sets (see \cite{MR1771072}). 
We will see that lost melodies 
are not computable from any \ITTM-random real. 
Moreover, we observe that as in \cite{MR2340241}, \ML-randomness is equivalent to a notion of incompressibility of the finite initial segments of the string. 

The first main result is an analogue to a result of Sacks \cite[Corollary 11.7.2]{MR2732288}: 
computability relative to all elements of a set of positive Lebesgue measure implies computability (asked in \cite[Section 3]{Infinite-computations}). 
This result is used in several proofs below. 

\begin{thm} \label{main theorem 1} (Theorem \ref{writable reals from non-null sets}) 
Suppose that $A$ is a subset of the Cantor space ${}^\omega2$ with $\mu(A)>0$ and a real $x$ is ITTM-computable from all elements of $A$. Then $x$ is ITTM-computable. 
\end{thm}

The proof rests on phenonema for infinite time computations that have no analogue in the context of Turing computability, in particular the  difference between writable, eventually writable and accidentally writable reals (see Definition \ref{definition: writable etc} or \cite{MR2493990}). 

We state some other main results. 
We obtain a variant for the stronger \emph{hypermachines} with $\Sigma_n$-limit rules \cite{FW} in Theorem \ref{variant for hypermachines}. 
We prove a variant of the previous theorem for recognizable sets.\footnote{An element $x$ of ${}^\omega 2$ is ITTM-recognizable if $\{x\}$ is ITTM-decidable (see Definition \ref{definition: recognizable}).} 
Thus we answer several questions posed in \cite[Section 5]{Ca5} and \cite[Section 6]{Randomness-and-degree-theory-for-ITRMs}. 

\begin{thm} \label{main theorem 1a} (Theorem \ref{recITTMsacks}) 
Suppose that $A$ is a subset of the Cantor space ${}^\omega2$ with $\mu(A)>0$ and a real $x$ is ITTM-recognizable from all elements of $A$. Then $x$ is ITTM-recognizable. 
\end{thm}


The next result characterizes ITTM-randomness by the values of an ordinal $\Sigma$ that is associated to ITTM-computations, the supremum of the ordinals coded by accidentally writable reals, i.e. reals that can be written on the tape at some time in some computation. 

\begin{thm} 
(Theorem \ref{characterization of ITTM-random}) 
A real $x$ is \ITTM-random if and only if it is random over $L_\Sigma$ and $\Sigma^x=\Sigma$. 
\end{thm} 

The following is a desirable property of randomness that holds for $\Pi^1_1$-randomness, but not for Martin-L\"of randomness. The property states that mutual randoms do not share non-computable information. Here, two reals are considered random if their join is random. 

\begin{thm} \label{computable from mutual ITTM-randoms} (Theorem \ref{computable from mutual ITTM-randoms})
If $x$ is computable from both $y$ and $z$, and $y$ and $z$ are mutual \ITTM-randoms, then $x$ is computable. 
\end{thm} 

We further analyze a decidable variant of ITTM-random that is analogous to $\Delta^1_1$-random.
We characterize this notion in Theorem \ref{characterization of decidable random} and prove an analogue to Theorem \ref{computable from mutual ITTM-randoms} and to van Lambalgen's theorem for this variant.  

All results in this paper, except for the Martin-L\"of variant in Section \ref{subsection ML}, work for Cohen reals instead of random reals, often with much simpler proofs, which we do not state explicitly. 

The main tool is a variant of random forcing suitable for models of weak set theories such as Kripke-Platek set theory. 
Previously, some results were formulated for the ideal of meager sets instead of the ideal of measure null sets, since the proofs use Cohen forcing and this is a set forcing in such models. 
Random forcing, on the other hand, is a class forcing in this contexts, 
and it is worthwhile to note that random generic is different from random over these models (see \cite{MR2860186}). 
These difficulties are overcome through an alternative definition of the forcing relation, which we call the \emph{quasi-forcing relation}. 

As a by-product, the analysis of random forcing allows more efficient proofs of classical results of higher recursion theory, such as Sacks' theorem that  $\{x\mid \omega_1^x>\om\}$ is a null set, 
although the quasi-generics used here are quite difference from generics used in forcing 
(see \cite[Remark after Theorem 6.6]{MR2860186}). 

We assume some familiarity with infinite time Turing machines (see \cite{MR1771072}), randomness (see \cite{MR2548883}) and admissible sets (see \cite{MR0424560}). 
In Section \ref{subsection ML} we will refer to some proofs in \cite[Section 3]{MR2340241} and \cite[Section 3]{Continuous-higher-randomness}. 
Moreover, we frequently use the Gandy-Spector theorem to represent $\Pi^1_1$ sets (see \cite[Theorem 5.5]{Hjorth-Vienna-notes-on-descriptive-set-theory}). 

The paper is structured as follows.  
In Section \ref{section random forcing}, we discuss random forcing over admissible sets and limits of admissible sets. 
In Section \ref{section ITTMs}, we prove results about infinite time Turing machines and computations from non-null sets. This includes the main theorem. 
In Section \ref{section random reals}, we use the previous results to prove desirable properties of randomness notions. 

We would like to thank Laurent Bienvenu for allowing us to include joint results with the first author on ITTM-genericity in Section \ref{subsection ITTM decidable random}. 
Moreover, we would like to thank Andre Nies, Philip Welch and Liang Yu for discussions related to the topic of this paper. 

\section{Random forcing over admissible sets} \label{section random forcing}

In this section, we present some results about random forcing over admissible sets and unions of admissible sets that are of independent interest. This is essential for the following proofs. The results simplify the approach to forcing over admissible sets (see \cite{MR1080970}) by avoiding a ranked forcing language. 

We first fix some (mostly standard) notation. 
A \emph{real} is a set of natural numbers or an element of the Cantor space ${}^\omega 2$. 
The basic open subsets of the Cantor space ${}^\omega 2$ will be denoted by $U_s=\{x\in {}^\omega 2\mid s\subseteq x\}$ for $s\in {}^{<\omega}2$. 
The Lebesgue measure on $2^\omega$ is the unique Borel measure $\mu$ with size $\mu(N_t)=2^{-|t|}$ for all $t\in 2^{<\omega}$. 
An \emph{admissible set} is a transitive set which satisfies Kripke-Platek set theory with the axiom of infinity. Moreover, an ordinal $\alpha$ is called admissible if $L_\alpha$ is admissible.

\subsection{The quasi-forcing relation} 

We work with the following version of random forcing. 
If $T$ is a subtree $T$ of ${}^{<\omega}2$, i.e. a downwards closed subset, let 
$$[T]= \{ x\in {}^\omega 2\mid \forall n\ x\upharpoonright n \in T\}$$ 
denote the set of (cofinal) branches of $T$. 
A \emph{perfect subtree} of $2^{<\omega}$ is a subtree without end nodes and cofinally many splitting nodes. 
We define random forcing as the set of perfect subtrees $T$ of ${}^{<\omega}2$ with $\mu([T])>0$, partially ordered by reverse incusion. 
Note that it can be easily shown (but will not be used here), for random forcing in any admissible set, that this partial order is dense in the set of Borel subsets $A$ of ${}^\omega 2$ (given by Borel codes) with $\mu(A)>0$. 
Note that random forcing is, in general, a class forcing over admissible sets, and this is the reason why we will need the following results. 

\begin{definition} 
Suppose that $\alpha$ is an ordinal and $x\in {}^\omega 2$. 
Then $x$ is \emph{random} over $L_\alpha$ if $x\in A$ for every Borel set $A$ with a Borel code in $L_\alpha$. 
\end{definition} 

We distinguish between the forcing relation for random forcing over an admissible and the \emph{quasi-forcing relation} defined below. 
In the definition of the quasi-forcing relation, the condition that a set is dense is replaced with the condition that the union of the conditions has full measure. 
Hence the quasi-forcing relation corresponds to a random real, i.e. a real which is a member of a class of definable sets of measure $1$, for instance all $\Pi^1_1$ sets of measure $1$. Such reals are sometimes called quasi-generics (see \cite{MR2601017}). 

This contrasts the notion of random generics in the sense of forcing. 
The following example shows that these two notions are different. 
Given any $n\geq 1$, we construct a dense $\Pi^1_1$ subset $A$ of the random forcing in $L_{\om}$ with $\mu(\bigcup A )<\frac{1}{n}$. 
Suppose that a $\Sigma_1$-definable enumeration $\langle B_\alpha\mid \alpha<\om\rangle$ of the Borel codes in $L_{\omega_1^\mathrm{CK}}$ for all Borel sets with positive measure and codes in $L_{\omega_1^\mathrm{CK}}$ given. We will use the same notation for a set and its code. 
Moreover, suppose that a partial surjection $f\colon \omega\rightarrow \om$ is given that is $\Sigma_1$-definable over 
$L_{\om}$. 
We define a sequence of Borel sets $A_\alpha\subseteq B_\alpha$ with $0<\mu(A_\alpha)<2^{-(i+n+1)}$, 
where $i$ is least such that $f(i)=\alpha$. 
Then  $A=\bigcup_{\alpha<\om} A_\alpha$ is a $\Pi^1_1$ set by the Gandy-Spector theorem \cite[Theorem 5.5]{Hjorth-Vienna-notes-on-descriptive-set-theory}. 
The difference is illustrated even better by Liang Yu's result \cite{MR2860186} that $\omega_1^x>\om$ for any random generic $x$ over $L_{\om}$. 
Together with a classical result (for a proof, see Lemma \ref{classical characterization of Pi11-randoms} below) shows that no random generic over $L_{\om}$ avoids every $\Pi^1_1$ null set. 

Moreover, it is shown below that the quasi-forcing relation for $\Delta_0$-formulas over admissible sets is definable, 
while we do not know if this holds for the forcing relation. 

%


We now define Boolean values for the quasi-forcing relation for random forcing. 
An $\infty$-Borel code is a set of ordinals that codes a set built from basic open subsets of ${}^{\omega}2$ and their complements by forming intersections and unions of any ordinal length. 
We will write $\bigvee_{i\in I} x_i$ for the canonical code for the union of the sets coded by $x_i$ for $i\in I$, and similarly for $\bigwedge_{i\in I} x_i$ and $\neg x$. 

\begin{definition} 
Suppose that $L_\alpha$ is admissible or an increasing union of admissible sets. 
We define $\llbracket \varphi(\sigma_0,\dots, \sigma_n) \rrbracket$ by induction in $L_\alpha$, where $\sigma_0,\dots, \sigma_n\in L_\alpha$ are names for random forcing and $\varphi(x_0, \dots, x_n)$ is a formula. 
\begin{enumerate-(a)} 
\item 
$\llbracket \sigma\in \tau\rrbracket=\bigvee_{(\nu,p)\in \tau} \llbracket \sigma=\nu\rrbracket \wedge p$. 
\item 
$\llbracket \sigma = \tau\rrbracket=(\bigwedge_{(\nu,p)\in \sigma} \llbracket \nu\in\tau\rrbracket \wedge p)\wedge
(\bigwedge_{(\nu,p)\in \tau} \llbracket \nu\in\sigma\rrbracket \wedge p)$. 
\item 
$\llbracket \exists x\in \sigma_0\ \varphi(x,\sigma_0,\dots, \sigma_n)\rrbracket= \bigvee_{(\nu,p)\in \sigma_0} \llbracket \varphi(\nu,\sigma_0,\dots, \sigma_n)\rrbracket \wedge p$. 
\item 
$\llbracket \neg \varphi(\sigma_0,\dots, \sigma_n)\rrbracket=\neg \llbracket \varphi(\sigma_0,\dots, \sigma_n)\rrbracket$. 
\item 
$\llbracket \exists x\ \varphi(x,\tau)\rrbracket=\bigcup_{\sigma\in L_\alpha} \llbracket \varphi(\sigma,\tau)\rrbracket$. 
\end{enumerate-(a)} 
\end{definition} 

We will identify $\llbracket \varphi(\sigma_0,\dots, \sigma_n)\rrbracket$ with the subset of ${}^\omega 2$ that it codes. 
This quasi-forcing relation is defined as follows. 

\begin{definition} 
Suppose that $\alpha$ is admissible or a limit of admissibles, $p$ a random condition in $L_\alpha$, $\varphi(x_0,\dots, x_n)$ formula and $\sigma_0,\dots ,\sigma_n$ random names in $L_\alpha$. 
We define $p\Vdash^{L_\alpha} \varphi(\sigma_0,\dots, \sigma_n)$ if $\mu([p]\setminus \llbracket \varphi(\sigma_0, \dots, \sigma_n)\rrbracket)=0$. 
\end{definition} 

\begin{lemma} 
Suppose that $\alpha$ is admissible or a limit of admissibles. 
Then the function which associates a Boolean value to $\Delta_0$-formulas $\varphi(\sigma_0,\dots ,\sigma_n)$ and the forcing relation for random forcing are $\Delta_1$-definable over $L_\alpha$.  
\end{lemma} 
\begin{proof} 
The Boolean values are defined by a $\Delta_1$-recursion and the measure corresponding to a code is definable by a $\Delta_1$-recursion. 
This implies that the forcing relation is $\Delta_1$-definable. 
\end{proof} 

\begin{definition} 
Suppose that $\alpha$ is an ordinal and $x\in {}^\omega 2$. 
We define $\sigma^x=\{\nu^x\mid (\nu,p)\in \sigma,\ x\in [p]\}$ for $\sigma\in L_\alpha$ by induction on the rank. 
\begin{enumerate-(a)} 
\item 
The generic extension of $L_\alpha$ by $x$ is defined as $L_\alpha[x]=\{\sigma^x\mid \sigma\in L_\alpha\}$. 
\item 
The $\alpha$th level of the $L$-hierarchy built over $x$, with $L_0[x]=\mathrm{tc}(\{x\})$, is denoted by $L_\alpha^x$. 
\end{enumerate-(a)} 
\end{definition} 

We will show in Lemma \ref{generic extension is subset of L stage} and Lemma \ref{L stage is subset of generic extension} that the sets $L_\alpha[x]$ and $L_\alpha^x$ are equal if $x$ is random over $L_\alpha$ and $\alpha$ is admissible or a limit of admissibles. 

\begin{lemma} \label{application of boolean values} 
Suppose that $L_\alpha$ is admissible or an increasing union of admissible sets, $\sigma_0,\dots, \sigma_n\in L_\alpha$ are names for random forcing, $\varphi(x_0, \dots, x_n)$ is a $\Delta_0$-formula and $x$ is random over $L_\alpha$. Then 
$$L_\alpha[x]\vDash \varphi(\sigma_0^x,\dots, \sigma_n^x)\Longleftrightarrow x\in \llbracket \varphi(\sigma_0,\dots, \sigma_n)\rrbracket.$$ 
Moreover, this holds for all formulas if $\alpha$ is countable in $L_\beta$ and $x$ is random over $L_\beta$. 
\end{lemma} 
\begin{proof} 
By induction on the ranks of names and on the formulas. 
\end{proof} 

The following is a version of the forcing theorem for the quasi-forcing relation. 

\begin{lemma} \label{forcing theorem} 
Suppose that $\alpha$ is admissible or a limit of admissibles and $p$ is a random condition in $L_\alpha$. 
Then $p\Vdash^{L_\alpha} \varphi(\sigma_0,\dots, \sigma_n)$ if and only if $L_\alpha[x]\vDash \varphi(\sigma_0^x,\dots, \sigma_n^x)$ for all random $x\in [p]$ over $L_\alpha$. 
\end{lemma} 
\begin{proof} 
Suppose that $p\Vdash^{L_\alpha} \varphi(\sigma_0,\dots, \sigma_n)$ and $x\in [p]$ is random over $L_\alpha$. 
Then $\mu([p]\setminus \llbracket \varphi(\sigma_0, \dots, \sigma_n)\rrbracket)=0$. 
Since $x$ is random over $L_\alpha$, $x\in \llbracket \varphi(\sigma_0, \dots, \sigma_n)\rrbracket$. 
This implies $L_\alpha[x]\vDash \varphi(\sigma_0^x,\dots, \sigma_n^x)$ by Lemma \ref{application of boolean values}. 

Suppose that $p\not\Vdash^{L_\alpha} \varphi(\sigma_0,\dots, \sigma_n)$. 
Then $\mu([p]\setminus \llbracket \varphi(\sigma_0, \dots, \sigma_n)\rrbracket)>0$. 
Suppose that $x\in [p]\setminus \llbracket \varphi(\sigma_0, \dots, \sigma_n)\rrbracket$ is random over $L_\alpha$. 
Then $L_\alpha[x]\vDash \neg\varphi(\sigma_0^x,\dots, \sigma_n^x)$ by Lemma \ref{application of boolean values}. 
\end{proof} 

The following is a version of the truth lemma for the quasi-forcing relation. 

\begin{lemma} \label{truth lemma} 
Suppose that $\alpha$ is admissible or a limit of admissibles and $x$ is random over $L_\alpha$. 
Then $L_\alpha[x]\vDash \varphi(\sigma^x)$ holds if and only if there is a random condition $p$ in $L_\alpha$ with $x\in [p]$ and $p\Vdash \varphi(\sigma)$. 
\end{lemma} 
\begin{proof} 
Suppose that $x\in [p]$ and $p\Vdash \varphi(\sigma)$. 
Then $\mu([p]\setminus \llbracket \varphi(\sigma)\rrbracket)=0$. 
Since $x$ is random over $L_\alpha$, $x\in L_\alpha[x]\vDash \varphi(\sigma^x)$ holds. 
Then $L_\alpha[x]\vDash \varphi(\sigma^x)$ by Lemma \ref{application of boolean values}. 

Suppose that $L_\alpha[x]\vDash \varphi(\sigma^x)$ holds. 
Then $x\in \llbracket \varphi(\sigma)\rrbracket$ by Lemma \ref{application of boolean values}. 
Since $\mu(\llbracket \varphi(\sigma)\rrbracket)$ is the supremum of $\mu([p])$, where $p$ is a condition in $L_\alpha$ with $[p]\subseteq \llbracket \varphi(\sigma)\rrbracket$, and $x$ is random over $L_\alpha$, there is a condition $p$ in $L_\alpha$ with $x\in [p]$. 
Since $[p]\subseteq \llbracket \varphi(\sigma)\rrbracket$,  $p\Vdash^{L_\alpha} \varphi(\sigma)$. 
\end{proof}

\subsection{The generic extension} 

If $\alpha$ is admissible or a limit of admissibles and $x$ is random over $L_\alpha$, we show that $L_\alpha[x]$ is equal to $L^x_\alpha$. 

\begin{lemma} \label{generic extension is subset of L stage} 
Suppose that $\alpha$ is admissible or a limit of admissibles and $x$ is random over $L_\alpha$. 
Then for all $\gamma<\alpha$ and all $\sigma\in L_{\gamma}$, $\sigma^x\in L_{\gamma+2}^x$. 
\end{lemma} 
\begin{proof} 
Suppose that $\sigma\in L_{\gamma}$ is a name. We define for all $\beta<\gamma$ the $\beta$-th approximate evaluation for $\sigma$ as the function 
$$f_{\beta,\sigma}:\text{tc}(\sigma)\cap L_{\beta}\rightarrow L_{\alpha}[x]$$
which maps $(\tau,p)$ to $\tau^{x}$ if $x\in [p]$ and to $\emptyset$ otherwise. Moreover, let $F_{\gamma}(\beta)=f_{\beta,\sigma}$ for $\beta<\gamma$. 

We will show by simultaneous induction that $f_{\beta,\sigma}, F_{\beta} \in L_{\beta+2}$ for all $\beta<\gamma$. 
It will then follow easily that $\sigma^{x}$ is definable over $L_{\gamma+1}$ and hence an element of $L_{\gamma+2}$. 
Suppose that $\beta=\theta+1$. Then $F_{\theta}\in L_{\theta+1}$ by the inductive hypothesis. We define $f_{\beta,\sigma}$ over $L_{\theta}$ by 
$$f_{\beta,\sigma}(\tau,p)=\{F_{\theta}(\bar{\tau})\mid x\in[p],\ \exists q\ ((\bar{\tau},q)\in\tau,\  x\in[q])\}$$
Then $f_{\beta,\sigma}\in L_\beta= L_{\theta+1}$. Let $F_{\beta}=F_{\theta}\cup\{(\theta,f_{\theta,\sigma})\}$. 
If  $\beta$ is a limit ordinal, we define $f_{\beta,\sigma}$ by 
$$f_{\beta,\sigma}((\tau,p))=\{F_{\delta}(\bar{\tau})\mid x\in[p],\ \delta<\gamma,\ \tau^{\prime}\in L_{\delta},\ \exists{q}((\bar{\tau},q)\in\tau,\  x\in [q])\}.$$ 
To define $F_{\beta}$ in the limit case, we proceed as follows. 
Note that for $\delta<\beta$, $F_{\delta}$ is the unique function which satisfies the following in $L_{\beta}$: 
$\mathrm{dom}(F)=\delta$, $F(0)=0$, $F$ is continuous at all limits $\gamma<\delta$, and $F(\eta+1)$ is defined as in the successor case above for all $\eta<\delta$. 
It follows that $f_{\gamma,\sigma}$ is definable over $L_{\gamma+1}$ and hence $\sigma^{x}=f_{\gamma,\sigma}(\sigma)\in L_{\gamma+2}$. 
\end{proof} 


\begin{lemma} \label{L stage is subset of generic extension} 
Suppose that $\alpha$ is admissible or a limit of admissibles and $x$ is random over $L_\alpha$. 
Then $L_{\alpha}^x\subseteq L_{\alpha}[x]$. 
\end{lemma} 
\begin{proof} 
It is sufficient to prove this for the case that $\alpha$ is admissible. 
It is sufficient to show that there is a $\Sigma_1$-definable sequence $\langle \tau_\gamma,\alpha_\gamma \mid \gamma<\alpha \rangle$ such that each $\tau_\gamma$ is a name, $\alpha_\gamma$ is an ordinal, $\sup_{\gamma<\alpha}\alpha_\gamma=\alpha$, $\tau_{\gamma}$ is uniformly $\Sigma_1$-definable over $L_{\alpha_\gamma}$ and $\tau^x=L_{\gamma}[x]$. Since $L_\gamma[x]$ is transitive, this implies the claim. 

Suppose that $\tau_\gamma$ and $\alpha_\gamma$ are defined. Suppose that $(\sigma,p)\in \tau_\gamma$. 
Let $\varphi^x$ denote the relativization of a formula $\varphi$ to a set $x$. 
Since $\alpha$ is admissible, there is a least ordinal
$\delta_{\sigma,p}$ such that $\llbracket \varphi^{\tau_\gamma} (\sigma_0,\dots, \sigma_n)\rrbracket \subseteq \delta_{\sigma,p}$ for all formulas $\varphi(x_0,\dots, x_n)$ and all names $\sigma_0,\dots,\sigma_n$ such that there are conditions $p_i$ with $(\sigma_i, p_i)\in \tau_\gamma$ for all $i\leq n$. 
 
Let $\alpha_{\gamma+1}=\sup_{(\sigma,p)\in \tau_\gamma}\delta_{\sigma,p}$. Then $\alpha_{\gamma+1}$ is uniformly $\Sigma_1$-definable from $\alpha_\gamma$ and $\tau_\gamma$. 
Moreover, let $\tau_\gamma^{\varphi}=\{(\sigma,p)\in \tau_\gamma\mid p\Vdash \phi^{\tau_\gamma}(\sigma))\}$. Then $\tau_\gamma^{\phi}$ is uniformly $\Sigma_1$-definable over $L_{\alpha_{\gamma+1}}$. 
By forming unions at limits, we define the sequence $\langle \tau_\gamma,\alpha_\gamma \mid \gamma<\alpha \rangle$ in a $\Sigma_1$ recursion. 
\end{proof} 

We now argue that $L_\alpha[x]$ is admissible if $\alpha$ is admissible and $x$ is sufficiently random.

\begin{lemma} \label{preservation of admissibility by randoms} 
Suppose that $\alpha$ is admissible or a limit of admissibles, and $x$ is random over $L_{\alpha+1}$. 
Then $L_\alpha[x]$ is admissible or a limit of admissibles, respectively. 
\end{lemma} 
\begin{proof} 
It is sufficient to prove this for the case where $\alpha$ is admissible. 
Suppose that $f$ is a $\Sigma_1$-definable function over $L_\alpha[x]$ that is cofinal in $\alpha$ and has domain $\eta<\alpha$. We will assume that $\eta=\omega$ to simplify the notation. 

Suppose that $\dot{x}$ is a name for the random generic 
and that $\dot{f}$ is a name for $f$. 
Since $f$ is a function in $L_\alpha[x]$ and $x$ is random over $L_{\alpha+1}$, 
$$\mu(\bigcap_{n\in\omega} \llbracket \exists \alpha\ \dot{f}(n)\in L_\alpha[\dot{x}]\rrbracket)>0,$$ 
where the Boolean value of existential formulas is defined as a union in the obvious way. 
Let $\mu(\bigcap_{n\in\omega} \llbracket \exists \alpha\ \dot{f}(n)\in L_\alpha[\dot{x}]\rrbracket)=\epsilon$. 

\begin{claim} \label{claim for preservation of admissibility} 
$\mu(\bigcap_{n\in\omega} \llbracket \exists \alpha\ \dot{f}(n)\in L_\alpha[\dot{x}]\rrbracket\setminus \llbracket \exists g\ \forall n\ (\dot{f}(n)=g(n)) \rrbracket)=0$. 
\end{claim} 
\begin{proof} 
Suppose that $\delta\leq\epsilon$ with $\delta\in \mathbb{Q}$. 
We consider the $\Delta_0$-definable function $h$ that maps $n$ to the least $\bar{\alpha}<\alpha$ such that 
$$\mu(\bigcap_{i\leq n} \llbracket \dot{f}(i)\in L_{\bar{\alpha}}[\dot{x}]\rrbracket)\geq\delta$$ 
and this $\Sigma_1$-statement (i.e. the statement that the measure is at least $\delta$) is witnessed in $L_{\bar{\alpha}}$. 
Since $\alpha$ is admissible, we obtain some $\gamma<\alpha$ with $\mu(\bigcap_{n\in\omega} \llbracket \exists \alpha\ \dot{f}(n)\in L_\gamma\rrbracket)\geq \delta$ 
and hence 
$$\mu(\bigcap_{n\in\omega} \llbracket \exists \alpha\ \dot{f}(n)\in L_\alpha\rrbracket\setminus \llbracket \exists g\ \forall n\ (\dot{f}(n)=g(n)) \rrbracket)\leq \epsilon - \delta.$$ 
\end{proof} 

Since the set in Claim \ref{claim for preservation of admissibility} is definable over $L_\alpha$, this implies the statement of Lemma \ref{preservation of admissibility by randoms} by Lemma \ref{application of boolean values}. 
\end{proof} 

As an example for how the previous can be applied to prove known theorems, we consider the following classical result (see [Theorem 9.3.9, Nies]). 
Note that random over $L_{\om}$ in our notation is equivalent to $\Delta^1_1$-random.

\begin{lemma} \label{classical characterization of Pi11-randoms} (see \cite[Theorem 9.3.9]{MR2548883}) 
A real $x$ is $\Pi^1_1$-random if and only if $x$ is $\Delta^1_1$-random and $\omega_1^x=\om$. 
\end{lemma} 
\begin{proof} 
We first claim that $\omega_1^x=\om$ for every $\Pi^1_1$-random real. 
The set of random reals over $L_{\alpha+1}$ has measure $1$, and for  these reals $x$, we have $\omega_1^x=\om$ by Lemma \ref{preservation of admissibility by randoms}. 
Moreover $\omega_1^x>\omega_1$ if and only if there is an admissible ordinal in $L_{\omega_1^x}[x]$, hence the set of these reals is $\Pi^1_1$ by the Gandy-Spector theorem \cite[Theorem 5.5]{Hjorth-Vienna-notes-on-descriptive-set-theory}. 
Thus $\omega_1^x=\om$. 

In the other direction, let $A$ denote the largest $\Pi^1_1$ null set (see \cite[Theorem 5.2]{MR2340241} and Section \ref{subsection ITTM-random reals} below). 
Then $A\subseteq \{x\mid \omega_1^x<\om\}\cup \bigcup_{\alpha<\om} A_\alpha$, where $A_\alpha$ is a Borel set with a code in $L_{\om}$, by the Gandy-Spector theorem \cite[Theorem 5.5]{Hjorth-Vienna-notes-on-descriptive-set-theory}. 
Since $A$ is the largest $\Pi^1_1$ null set, equality holds. 
If $x$ is $\Delta^1_1$-random and $\omega_1^x=\om$, then $x\notin A_\alpha$ for all $\alpha<\om$ and hence $x\notin A$. 
\end{proof} 

\subsection{Side-by-side randoms} 

Two reals $x,y$ are \emph{side-by-side random} over $L_\alpha$ if $\langle x,y\rangle$ is random over $L_\alpha$ for the Lebesgue measure on ${}^{\omega}2\times{}^{\omega}2$. 
The following Lemma \ref{intersection from mutually randoms} is analogous to known results for arbitrary forcings over models of set theory, however the classical proof does not work in our setting. 

\begin{lemma} 
Suppose that $x,y$ are side-by-side random over $L_\alpha$. Then $x$ is random over $L_\alpha$. 
\end{lemma} 
\begin{proof} 
Suppose that $A$ is a Borel subset of ${}^{\omega}2$ with Borel code in $L_\alpha$. Then $\langle x,y\rangle\in A\times {}^{\omega}2$. Hence $x\in A$. 
\end{proof} 

\begin{lemma}\label{application of koenigs lemma} 
Suppose that $\langle A_s \mid s\in {}^{<\omega}2 \rangle$ is a sequence of Lebesgue measurable subsets of ${}^\omega 2$ such that 
$A_t\subseteq A_s$ for all $s\subseteq t$ in ${}^n 2$ 
and $\mu(\bigcap_n A_{x\upharpoonright n})=0$ for all $x\in {}^{\omega}2$. 
Then for every $\epsilon>0$, there is some $n$ such that for all $s\in {}^n 2$, $\mu(A_s)<\epsilon$. 
\end{lemma} 
\begin{proof} 
If the claim fails, then the tree $T=\{s\in {}^{<\omega}2\mid \mu(A_s)\geq\epsilon\}$ is infinite. 
By K\"onig's lemma, $T$ has an infinite branch $x\in {}^{\omega}2$. 
Then $\mu(\bigcap_n A_{x\upharpoonright n})\geq\epsilon$, contradicting the assumption. 
\end{proof} 

We use the forcing theorem for random forcing over admissible sets $L_\alpha$ to prove an analogue to the fact that the intersection of mutually generic extensions is equal to the ground model. 

\begin{lemma} \label{intersection from mutually randoms} 
Suppose that $L_\alpha$ is admissible or an increasing union of admissible sets and that $x,y$ are 
side-by-side random over $L_\alpha$. Then $L_\alpha[x]\cap L_\alpha[y]=L_\alpha$. 
\end{lemma} 
\begin{proof} 
Let $\mathbb{P}$ denote the random forcing on ${}^{\omega}2$ in $L_\alpha$ and $\mathbb{Q}$ the random forcing on ${}^{\omega}2\times{}^{\omega}2$ in $L_\alpha$. 
Suppose that $z\in L_\alpha[x]\cap L_\alpha[y]$. 
Moreover, suppose that $\dot{x}, \dot{y}$ are $\mathbb{P}$-names for $z$ with $\dot{x}^x=z$ and $\dot{y}^y=z$. 
We can assume that $\dot{x},\dot{y}$ are $\mathbb{Q}$-names by identifying them with the $\mathbb{Q}$-names induced by $\dot{x}$, $\dot{y}$. 
Then every Borel subset of ${}^{\omega}2$ that occurs in $\dot{x}$ is of the form $A\times {}^{\omega}2$ and every Borel subset of ${}^{\omega}2$ occuring in $\dot{y}$ is of the form ${}^{\omega}2\times A$. 




\begin{claim} 
No condition $p$ forces over $L_\alpha$ that $\dot{x}=\dot{y}$. 
\end{claim} 
\begin{proof} 
Suppose that $p\Vdash \dot{x}=\dot{y}$ and $\mu([p])\geq\epsilon>0$. 
Then $p\Vdash \bigvee_{s\in {}^k 2} \dot{x}\upharpoonright k=\dot{y}\upharpoonright k=s$ by Lemma \ref{forcing theorem}. 
Let $A_s=\llbracket \dot{x}\upharpoonright k=s\rrbracket$ and $B_s=\llbracket \dot{y}\upharpoonright k=s\rrbracket$. 
Then $\mu([p]\setminus \bigcup_{s\in {}^n 2} (A_s \times B_s))=0$ by Lemma \ref{application of boolean values}. 

There is some $n$ such that $\mu(A_s)<\epsilon$ for all $s\in {}^n 2$ by Lemma \ref{application of koenigs lemma}. 
Since $\sum_{s\in {}^n 2} \mu(B_s)=1$, 
$\sum_{s\in {}^n 2} \mu(A_s)\mu(B_s)<\epsilon$. 
The assumption $p\Vdash \dot{x}=\dot{y}$ implies that $\mu([p]\setminus \bigcup_{s\in {}^n 2} A_s\times B_s)=0$. 
Hence $\mu([p])\leq \mu(\sum_{s\in {}^n 2} \mu(A_s)\mu(B_s))<\epsilon $, 
contradicting the assumption that $\mu([p])\geq \epsilon$. 
\end{proof} 
This completes the proof of Lemma \ref{intersection from mutually randoms}. 
\end{proof}

\section{Computations from non-null sets} \label{section ITTMs} 

In this section, we prove an analogue to the following result of Sacks: any real that is computable from all elements of a set of positive measure is itself computable. This is essential to analyze randomness notions later.

\subsection{Facts about infinite time Turing machines} \label{subsection intro to ITTMs}

An infinite time Turing machine (ITTM) is a Turing machine that is allowed to run for an arbitrary ordinal time, with the rule of forming the inferior limit in each tape cell and of the (numbered) states in each limit step of the computation. 
The inputs and outputs of such machines are reals. 

We recall some basic facts about these machines (see \cite{MR1771072, MR2493990}). 
The computable sequences are here called \emph{writable} to distinguish this from the following concepts of computability. 
These notions from \cite{MR1771072} are interesting on their own and will be essential in the following proofs via results in \cite{MR2493990}. 

\begin{defini}(See \cite{MR1771072}) \label{definition: writable etc} 
\begin{enumerate-(a)} 
\item 
A real $x$ is \emph{writable} (or \emph{computable}) if and only if there is an ITTM-program $P$ such that $P$, when run on the empty input, halts with $x$ written on the output tape.
\item 
A real $x$ is \emph{eventually writable} if and only if there is an ITTM-program $P$ such that $P$, when run on the empty input, has from some point of time on $x$ written on the output tape and never changes
the content of the output tape from this time on.
\item 
A real $x$ is \emph{accidentally writable} if and only if there is an ITTM-program $P$ such that $P$, when run with empty input, has $x$ written on the output tape at some time (but may overwrite this later on).
\end{enumerate-(a)} 
\end{defini}

We write $P^x\downarrow=i$ if $P^x$ halts with output $i$. 
The notation $\Sigma_n$ will always refer to the standard Levy hierarchy, obtained by counting the number of quantifier changes around a $\Delta_0$ kernel. 

The ordinal $\lambda$ is defined as the supremum of the halting times of ITTM-computations (i.e. the \emph{clockable ordinals}), and equivalently \cite[Theorem 1.1]{MR1734198} the supremum of the writable ordinals, i.e. the ordinals coded by writable reals. 
Moreover, $\zeta$ is defined as the supremum of the eventually writable ordinals, and $\Sigma$ is the supremum of the accidentally writable ordinals. 
The ordinals $\lambda^x$, $\zeta^x$ and $\Sigma^x$ are defined relative to an oracle $x$. 

We will use the following theorem by Welch \cite[Theorem 1, Corollary 2]{MR2493990}. 

\begin{thm} \label{lambdazetasigma} (see \cite[Theorem 1, Corollary 2]{MR2493990}) 
Suppose that $y$ is a real. 
Then $\lambda^{y},\zeta^{y},\Sigma^{y}$ have the following properties. 
\begin{enumerate-(1)} 
\item 
$L_{\zeta^{y}}[y]$ is the set of writable reals in $y$. 
\item 
$L_{\zeta^{y}}[y]$ is the set of eventually writable reals in $y$. 
\item 
$L_{\Sigma^{y}}[y]$ is the set of accidentally  reals in $y$. 
\end{enumerate-(1)} 
Moreover $(\lambda^{y},\zeta^{y},\Sigma^{y})$ is the lexically minimal triple of ordinals with 
$$L_{\lambda^{y}}[y]\prec_{\Sigma_{1}}L_{\zeta^{y}}[y]\prec_{\Sigma_{2}}L_{\Sigma^{y}}[y].$$ 

\end{thm}

It is worthwhile to note that the precise definition of the Levy hierarchy is important for the reflection in Theorem \ref{lambdazetasigma}. 
The characterization of $\lambda$, $\zeta$ and $\Sigma$ 
fails if we allow arbitrary additional bounded quantifiers in the Levy hierarchy, since this variant of $\Sigma_2$-formulas allows to express the fact that a set is admissible. However, $L_\zeta$ is admissible \cite[Fact 2.2]{MR2493990}, but $L_\Sigma$ is not admissible \cite[Lemma 6]{MR2493990}. 

We will also use the following information about $\lambda$, $\zeta$ and $\Sigma$. 

\begin{thm}\label{ITTMcharacteristics} 
\begin{enumerate-(a)} 
\item 
If the output of an ITTM-program $P$ stabilizes, then it stabilizes before time $\zeta$. 
\item 
All non-halting ITTM-computations loop from time $\Sigma$ on. 
\item 
$\lambda$ and $\zeta$ are admissible limits of admissible ordinals 
(and more). 
\item \label{ITTMcharacteristics every set is countable} 
In $L_\lambda$ every set is countable, and the same holds for $L_\zeta$ and $L_\Sigma$. 
\end{enumerate-(a)} 
Moreover, all of these statements relativize to oracles. 
\end{thm} 
The proofs can be found in \cite{MR1771072, MR2493990}. 
We will write $x\leq_{w}y$, $x\leq_{ew}y$, $y\leq_{aw}y$ to indicate that $x$ is writable, eventually writable or accidentally writable, respectively, in the oracle $y$.

\begin{lemma} \label{characterization of ITTM-semidecidable} 
The following are equivalent for a subset $A$ of ${}^\omega 2$. 
\begin{enumerate-(a)} 
\item 
$A$ is ITTM-semidecidable. 
\item 
There is $\Sigma_1$-formula $\varphi(x)$ such that for all $x\in {}^\omega 2$, $x\in A$ if and only $L_{\lambda^x}[x]\models \varphi(x)$. 
\end{enumerate-(a)} 
\end{lemma} 
\begin{proof} 
In the forward direction, the $\Sigma_1$-formula simply states the existence of a halting computation. 
In the other direction, we can search for a writable code for an initials segment of $L_{\lambda^x}[x]$ which satisfies $\varphi(x)$, using the fact that every set in $L_{\lambda^x}[x]$ has a writable code in $x$ by Theorem \ref{lambdazetasigma}. 
\end{proof} 

We call a subset of $2^{<\omega}$ \emph{enumerable} if there is an \ITTM\ listing its elements. 
It follows from Lemma \ref{characterization of ITTM-semidecidable} that it is equivalent for a subset $A$ of $2^{<\omega}$ that $A$ is semidecidable, $A$ is enumerable or that $A$ is $\Sigma_1$-definable over $L_{\lambda}$. 

Note that every \ITTM-semidecidable set is absolutely $\Delta^1_2$, i.e. it remains $\Delta^1_2$ with the same definition in any inner model and in any forcing extension. Therefore such sets are Lebesgue measurable and have the property of Baire by \cite[Exercise 14.4]{MR2731169}.

\subsection{Preserving reflection properties by random forcing} \label{subsection reflection} 

The following reflection argument is an essential step in the proof of the preservation of $\lambda$, $\zeta$ and $\Sigma$ (see Section \ref{subsection writable reals relative to many oracles} below) with respect to random forcing. 
We show that for admissibles or limits of admissibles $\alpha<\beta$, the statements $L_\alpha\prec_{\Sigma_1} L_\beta$ and $L_\alpha\prec_{\Sigma_2} L_\beta$ are preserved to generic extensions by sufficiently random reals (i.e. $L_\alpha[x] \prec_{\Sigma_n} L_\beta[x]$ holds for all sufficiently random reals $x$). 

\begin{definition} 
Suppose that $A$ is a Lebesgue measurable subset of ${}^\omega 2$. 
An element $x$ of ${}^\omega 2$ is a \emph{(Lebesgue) density point of $A$} if $\lim_n \frac{\mu(A\cap U_{x\upharpoonright n})}{\mu(U_{x\upharpoonright n})}=1$. 
Let $D(A)$ denote the set of density points of $A$. 
\end{definition} 

We will often use the following version of Lebesgue's density theorem. 
\begin{thm} \label{Lebesgue density theorem} (Lebesgue, see \cite[Section 8]{MR3003923}) 
If $A$ is any Lebesgue measurable subset of ${}^\omega 2$, 
then $\mu(A\triangle D(A))=0$. 
\end{thm} 

We now prove $\Sigma_1$ reflection and then $\Sigma_2$-reflection in random extensions, from a stronger hypothhesis. 

\begin{lemma}  \label{Sigma1 reflection} 
Suppose that $\alpha<\beta$, $\beta$ is 
admissible or a limit of admissibles and $L_\alpha\prec_{\Sigma_1} L_\beta$. 
If $x$ is random over $L_\beta$, then $L_\alpha[x]\prec_{\Sigma_1} L_\beta[x]$. 
\end{lemma} 
\begin{proof} 
Note that the assumption $L_\alpha\prec_{\Sigma_1} L_\beta$ implies that $L_\alpha$ is admissible. 
To see this, note that for any $\Sigma_1$-definable function $f\colon z\rightarrow L_\alpha$ over $L_\alpha$, the set $L_\alpha$ is a witness for the $\Sigma_1$-collection scheme for $f$ in $L_\beta$. 
It follows from the assumption $L_\alpha\prec_{\Sigma_1} L_\beta$ that 
there is a set in $L_\alpha$ witnessing the $\Sigma_1$-collection scheme for $f$ in $L_\alpha$, and in particuar $f\in L_\alpha$. 

Suppose that $L_\beta[x]\vDash \exists v\ \varphi(v,\tau^x)$, where $\varphi(v,w)$ is a $\Delta_0$-formula and $\tau\in L_\alpha$. 
We choose a witness $y\in L_\beta$ such that $\varphi(y,\tau^x)$ holds in $L_\beta[x]$. 
Moreover, suppose that $\sigma\in L_\beta$ is a name with $\sigma^x=y$. 
Let $A=\llbracket \varphi(\sigma,\tau) \rrbracket$ and 
let $A_n$ denote the set of $s\in {}^{<\omega}2$ with $\frac{\mu(A\cap U_s)}{\mu(U_s)}>1-2^{-n}$, for $n\in\omega$. 

In the next claim, we conclude from the Lebesgue density theorem \ref{Lebesgue density theorem} that $A$ is almost everywhere covered by the sets $U_s$ for $s\in A_n$. 
By an \emph{antichain} in $2^{<\omega}$ we mean a subset of $2^{<\omega}$ whose elements are pairwise incomparable with respect to $\subseteq$. 
Moreover, an \emph{antichain} in a subset $C$ of $2^{<\omega}$ is an antichain $\bar{C}\subseteq C$. 
A \emph{maximal antichain in $C$} is maximal with respect to $\subseteq$ among all antichains in $C$. 

\begin{claim} \label{measure of covering} 
If $A^\star$ is a maximal antichain in $A_n$, then $\mu(A\cap \bigcup_{s\in A^\star} U_s)=\mu(A)$.  
\end{claim} 
\begin{proof} 
Suppose that the claim fails and hence $\mu(A\setminus \bigcup_{s\in A^\star} U_s)>0$. 
Then there is a density point $z$ of $A\setminus \bigcup_{s\in A^\star} U_s$ by the Lebesgue density theorem \ref{Lebesgue density theorem}. 
Hence there is some $k$ with $\frac{\mu(A\cap U_{z\upharpoonright k})}{\mu(U_{z\upharpoonright k})}>1-2^{-n}$ and 
thus $t:= z{\upharpoonright}k\in A_n$, by the definition of $A_n$. 
However, $t$ is incomparable with all elements of $A^\star$, since $z\notin \bigcup_{s\in A^\star} U_s$. 
This contradicts the assumption that $A^\star$ is maximal. 
\end{proof} 

We choose a maximal antichain $A^\star_n$ in $A_n$ for each $n$. Since $A$ has a Borel code in $L_\beta$, we can choose $A^\star_n$ such that the sequence $\langle A^\star_n\mid n\in\omega \rangle$ is an element of $L_\beta$. 

We now aim to reflect the $\Sigma_1$-statement $\exists v\ \varphi(v,\tau)$ from $L_\beta[x]$ to $L_\alpha[x]$. 
Note that we do not have $\sigma$ and $A$ available in $L_\alpha$, but will instead obtain a name in $L_\alpha$ from $\sigma$ by reflection (i.e. by using the assumption that $L_\alpha\prec_{\Sigma_1} L_\beta$). 
The following argument ensures that there is in fact a subset $B$ of $A$ in $L_\beta$ with full measure relative to $A$ which witnesses the reflection, i.e. for randoms  in $B$ over $L_\beta$, the statement reflects. 

Suppose that $s\in A_n$ is given. 
We consider the $\Sigma_1$-formula $\psi_n(s)$ which states that there is a condition $p$ such that $[p]\subseteq U_s$, $\frac{\mu([p]\cap U_s)}{\mu(U_s)}>1-2^{-n}$ and $\exists \nu \ (p\Vdash \varphi(\nu,\tau))$. 
Since $s\in A_n$, $\psi_n(s)$  holds in $L_\beta$, and therefore in $L_\alpha$, by the assumption 
$L_\alpha\prec_{\Sigma_1} L_\beta$. 

Let $p_s^n$ denote the $<_L$-least condition in $L_\alpha$ witnessing $\psi_n(s)$ (in fact any choice would work, as long as the sequence $\langle p_s^n \mid n\in\omega \rangle$ is an element of $L_\beta$). 
Let $B_n=\bigcup_{s\in A^\star_{2n}} [p_s^{2n}]$ 
and $B=\bigcup_{n\in\omega} B_n$.

\begin{claim} 
$\mu(A\setminus B)=0$. 
\end{claim} 
\begin{proof} 
We have $\frac{\mu([p_s^{2n}]\cap U_s)}{\mu(U_s)}>1-2^{-2n}$  for all $s\in A_{2n}$ by the choice of $p_s^{2n}$, and 
$\frac{\mu(A\cap U_s)}{\mu(U_s)}>1-2^{-2n}$ for all $s\in A_{2n}$ by the definition of $A_{2n}$. 
Hence $\frac{\mu(A\cap B_n\cap U_s)}{\mu(U_s)}>1-2^{-n}$ for all $s\in A_{2n}$, by the definition of $B_n$. 
Therefore 
$$\frac{\mu(A\cap B_n\cap U_s)}{\mu(A\cap U_s)}\geq \frac{\mu(A\cap B_n\cap U_s)}{\mu(U_s)}>1-2^{-n}.$$ 
Moreover
$$\mu(\bigcup_{s\in A^\star_n} (A\cap U_s))=\mu(A\cap \bigcup_{s\in A^\star_n} U_s)=\mu(A)$$
by Claim \ref{measure of covering}. 
Hence the sets $U_s$ for $s\in A^\star_{2n}$ partition $A$ up to a null set. 
By applying the previous inequality separately for each $s\in A^\star_{2n}$, we obtain 
$\frac{\mu(A\cap B_n)}{\mu(A)} >1-2^{-n}$. 
Hence $\frac{\mu(A\cap B)}{\mu(A)}=1$ and $\mu(A\setminus B)=0$. 
\end{proof} 

Since $A$ has a Borel code in $L_\beta$ and therefore $\langle A_n\mid n\in\omega \rangle$ 
is an element of $L_\beta$, there is a sequence $\langle b_n\mid n\in\omega \rangle\in L_\beta$ such that $b_n$ is a Borel code for $B_n$. 
Therefore $B=\bigcup_{n\in\omega} B_n$ has a Borel code in $L_\beta$. 

\begin{claim} 
$L_\alpha[x]\vDash \exists v\ \varphi(v,\tau^x)$. 
\end{claim} 
\begin{proof} 
Recall that $\varphi(y,\tau^x)$ holds in $L_\beta[x]$ and $A=\llbracket \varphi(\sigma,\tau) \rrbracket$, therefore $x\in A$ by Lemma \ref{application of boolean values}. 
Since $x$ is random over $L_\beta$ by the assumption, and we have already proved that $\mu(A\triangle B)=0$, we have 
$x\in B$. 
Then there is some $n$ with $x\in B_n$. 
By the definition of $B_n$, there is some $s\in A_{2n}$ with $x\in [p_s^{2n}]\cap A$. 
By the definition of $p_s^{2n}$, there is a name $\nu\in L_\alpha$ such that $p_s^{2n}\Vdash^{L_\alpha} \varphi(\nu,\tau)$. 
Since $x\in [p_s^{2n}]$, Lemma \ref{forcing theorem} implies that $L_\alpha[x]\vDash \varphi(\nu^x,\tau)$. 
\end{proof} 

Hence the statement $\exists v\ \varphi(v,\tau^x)$ reflects to $L_\alpha[x]$. 
\end{proof}

We now move to the preservation of $\Sigma_n$-reflection 
under an appropriate hypothesis. 
The next result shows that the statement $L_\alpha\prec_{\Sigma_n} L_\beta$ is preserved for sufficiently random reals $x$, i.e. $L_\alpha[x]\prec_{\Sigma_n} L_\beta[x]$ holds in the generic extension. 
We first need the following lemma. 

\begin{lemma} \label{complexity of measure of Boolean values} 
Suppose that $\alpha$ is admissible or a limit of admissible ordinals, $t\in {}^{<\omega}2$, $\sigma\in L_\alpha$, $\epsilon\in \mathbb{Q}$, 
$n\geq 1$ and $\varphi$ is a formula. The formulas in the following claims have the parameters $t$, $\sigma$ and $\epsilon$. Let $m_{\sigma,t}=\mu(\llbracket \varphi(\sigma)\rrbracket\cap U_t)$. 
\begin{enumerate-(1)} 
\item 
If $\varphi$ is $\Sigma_n$, then 
\begin{enumerate-(a)} 
\item 
$m_{\sigma,t}>\epsilon$ is equivalent to a $\Sigma_n$-formula. 
\item 
$m_{\sigma,t}\leq\epsilon$ is equivalent to a $\Pi_n$-formula. 
\end{enumerate-(a)} 
\item 
If $\varphi$ is $\Pi_n$, then 
\begin{enumerate-(a)} 
\item 
$m_{\sigma,t}<\epsilon$ is equivalent to a $\Pi_n$-formula. 
\item 
$m_{\sigma,t}\geq\epsilon$ is equivalent to a $\Sigma_n$-formula. 
\end{enumerate-(a)} 
\end{enumerate-(1)} 
\end{lemma} 
\begin{proof} 
For $\Delta_0$-formulas $\varphi$, the claim holds since the function mapping $\sigma$ to $\llbracket \varphi(\sigma)\rrbracket$ is $\Delta_1$-definable in $\sigma$. 

Suppose that $\varphi(x,y)$ is a $\Pi_n$-formula. We aim to prove the first claim for the formula $\exists x \varphi(x,y)$. 

We have $\mu(\llbracket \exists x \varphi(x,y)\rrbracket\cap U_t)>\epsilon$ if and only if there is some $k$ and some $\sigma_0,\dots,\sigma_k$ such that $\mu(\llbracket \bigvee_{i\leq k} \varphi(\sigma_i,\tau)\rrbracket\cap U_t)>\epsilon$. By the Lebesgue density theorem \ref{Lebesgue density theorem}, the last inequality is equivalent to the statement that there is some $l$, a sequence $t_0,\dots, t_l$ of pairwise incompatible extensions of $t$ and some $\epsilon_0,\dots, \epsilon_l\in\mathbb{Q}$ such that $\epsilon=\sum_{i\leq k} \epsilon_i$ and for all $j\leq l$, there is some $i\leq k$ such that $\mu(\llbracket \varphi(\sigma_i,y)\rrbracket\cap U_{t_i})>\epsilon_i$. Using a universal $\Sigma_n$-formula, we obtain an equivalent $\Sigma_n$-statement. 

We have $\mu(\llbracket \exists x \varphi(x,y)\rrbracket\cap U_t)\leq \epsilon$ if and only if for all $\sigma_0,\dots, \sigma_k$, $\mu(\llbracket \bigvee_{i\leq k}\varphi(\sigma_i,\tau)\rrbracket)\leq\epsilon$. This is a $\Pi_n$-statement by argument in the previous case. 

The second claim follows by switching to negations. 
\end{proof} 

\begin{lemma} \label{Sigma2 reflection} 
Suppose that $\alpha<\beta$, $\beta$ is admissible or a limit of admissibles, $n\geq1$ and $L_\alpha\prec_{\Sigma_n} L_\beta$. 
Suppose that $\beta$ is countable in $L_\gamma$ and that $x$ is random over $L_{\gamma}$. 
Then $L_\alpha[x]\prec_{\Sigma_n} L_\beta[x]$. 
\end{lemma} 
\begin{proof} 
Note that the assumption $L_\alpha\prec_{\Sigma_1} L_\beta$ implies that $L_\alpha$ is admissible, 
as in the proof of Lemma \ref{Sigma1 reflection}. 

Suppose that the statement $\exists u\ \varphi(u,\tau^x)$ holds in $L_\beta[x]$, where $n=m+1$, $\varphi$ is $\Pi_m$ and $\tau\in L_\alpha$. 
Suppose that $\sigma_0$ is a name in $L_\beta$ with $L_\beta[x]\vDash \varphi(\sigma_0^x, \tau^x)$. 

Let $A=\llbracket \varphi(\sigma_0,\tau) \rrbracket$. 
Since $\beta$ is countable in $L_\gamma$, $A$ has a Borel code in $L_\gamma$. 
It follows from Lemma \ref{forcing theorem} that $x\in A$ and $\mu(A)>0$. 
Let $A_n$ denote the set of $s\in {}^ {<\omega}2$ such that 
$$\frac{\mu(A \cap U_s)}{\mu(U_{s})}> 1- 2^{-n}.$$ 

\begin{claim} \label{measure of covering for sigma 2} 
Suppose that $A^\star$ is a maximal antichain in $A_n$. 
Then 
$\mu(A\cap \bigcup_{s\in A^\star} U_s)=\mu(A)$. 
\end{claim} 
\begin{proof} 
The proof is identical to the proof of Claim \ref{measure of covering} 
via the Lebesgue density theorem \ref{Lebesgue density theorem}. 
\end{proof} 

We choose a maximal antichain $A^\star_n$ in $A_n$ for each $n$. Since $A$ has a Borel code in $L_{\beta+1}\subseteq L_\gamma$, it is possible to choose $A^\star_n$ such that the sequence $\langle A^\star_n\mid n\in\omega \rangle$ is an element of $L_\gamma$. 

Let $B_\sigma=\llbracket \varphi(\sigma,\tau) \rrbracket$. Then $A=B_{\sigma_0}$. 
We consider the statement $\psi_n(s)$ stating that there is some name $\sigma$ such that $\frac{\mu(B_\sigma\cap U_s)}{\mu(U_s)}>1-2^{-n}$. This is a $\Sigma_n$-statement by Lemma \ref{complexity of measure of Boolean values}. 

Since $s\in A_n$, $\psi_n(s)$ holds in $L_\beta$. Since $L_\alpha\prec_{\Sigma_n} L_\beta$, this implies that $\psi_n(s)$ holds in $L_\alpha$. 
Let $\sigma_s^n$ denote the $<_L$-least name in $L_\alpha$ witnessing $\psi_n(s)$, for $s\in A_n$ (in fact any choice would work, as long as the sequence $\langle \sigma_s^n \mid n\in\omega \rangle$ is an element of $L_\beta$). 

Let $B_n=\bigcup_{s\in A^\star_{2n}} B_{\sigma^{2n}_s}$ and $B=\bigcup_n B_n$. 
Since $\beta$ is countable in $L_\gamma$, $\langle \sigma_s^n\mid n\in\omega\rangle$ is an element of $L_\beta$ for each $s\in 2^{<\omega}$ and the sets $B_\sigma$ have Borel codes in $L_\gamma$ for all names $\sigma\in L_\beta$, uniformly in $\sigma$, the set $B$ has a Borel code in $L_\gamma$. 

\begin{claim} \label{A minus B is null} 
$\mu(A \setminus B)=0$. 
\end{claim} 
\begin{proof} 
We have $\frac{\mu(A\cap U_s)}{\mu(U_s)}>1-2^{-2n}$ for all $s\in A_{2n}^\star$ by the definition of $A_{2n}$ and $\frac{\mu(B_s\cap U_s)}{\mu(U_s)}> 1-2^{-2n}$ for all $s\in A_{2n}$ by the choice of $\sigma^{2n}_s$. Hence 
$$\frac{\mu(A\cap B_s)}{\mu(A\cap U_s)}\geq \frac{\mu(A\cap B_s)}{\mu(U_s)}>1-2^{-n}$$ 
for all $s\in A_{2n}$. 
Moreover, 
$$\mu(\bigcup_{s\in A^\star_n} (A\cap U_s))=\mu(A\cap \bigcup_{s\in A^\star_n} U_s)=\mu(A)$$
by Claim \ref{measure of covering for sigma 2}.
Since $A_n^\star\subseteq A_n$ is an antichain, the sets $A \cap U_s$ for $s\in A_n^\star$ are pairwise disjoint.
Therefore the previous inequality implies that 
$$\frac{\mu(A\cap B_n)}{\mu(A)}>1-2^{-n}.$$ 
Since $B=\bigcup_n B_n$, this implies $\frac{\mu(A\cap B)}{\mu(A)}=1$ and hence $\mu(A\setminus B)=0$. 
\end{proof} 

\begin{claim} \label{Sigma2 statement reflects} 
$\varphi((\sigma_s^{2n})^x,\tau^x)$ holds in $L_\alpha[x]$. 
\end{claim} 
\begin{proof} 
We have $x\in A$ by the assumption. Since $A$ and $B$ have Borel codes in $L_\gamma$, $\mu(A\setminus B)=0$ and $x$ is random over $L_\gamma$, $x\in B$. Then $x\in B_n$ for some $n$ and $x\in B_{\sigma^{2n}_s}=\llbracket \varphi(\sigma^{2n}_s,\tau)\rrbracket$ for some $s\in A^\star_{2n}$. 
By Lemma \ref{forcing theorem}, $\varphi((\sigma^{2n}_s)^x,\tau)$ holds in $L_\alpha[x]$. 
\end{proof} 

Hence the statement $\exists u\ \varphi(u,\tau^x)$ reflects to $L_\alpha[x]$. 
\end{proof} 

The assumptions in Lemma \ref{Sigma2 reflection} for $n=2$ are not optimal for the application to \ITTM s below. 
We will see in Section \ref{subsection ITTM-random reals} that ITTM-randomness is a sufficient assumption for the applications. 

\subsection{Writable reals from non-null sets} \label{subsection writable reals relative to many oracles} 

We will prove an analogue to the following theorem for infinite time Turing machines. Let $\leq_{\text{T}}$ denote Turing reducibility. 

\begin{thm} (Sacks, see \cite[Corollary 11.7.2]{MR2732288}) \label{writable reals from non-null sets} 
If a real $x$ is computable if and only if $\{y \mid x \leq_{\text{T}}y\}$ has positive Lebesgue measure. 
\end{thm}

In \cite{Infinite-computations}, analogues of this theorem for other machines were considered. 
It was asked if this holds for infinite time Turing machines, and this was only proved for non-meager Borel sets, via Cohen forcing over levels of the constructible hierarchy. 
With the results in Section \ref{section random forcing}, we prove this for Lebesgue measure. 

\begin{thm} \label{ITTMSacks} 
\begin{enumerate-(1)}
\item 
A real $x$ is writable if and only if $\mu(\{y:x\leq_{w}y\})>0$
\item 
A real $x$ is eventually writable if and only if $\mu(\{y:x\leq_{ew}y\})>0$
\item 
A real $x$ is accidentally writable if and only if $\mu(\{y:x\leq_{aw}y\})>0$
\end{enumerate-(1)}
\end{thm}
\begin{proof} 
The forward direction is clear in each case. In the other direction, we only prove the writable case, since the proofs of the remaining cases are analogous. 

Let $W_{x}:=\{y:x\leq_{w}y\}$ and choose some sufficiently random $r\in W_x$. 
Since $\Sigma$ is a limit of admissible ordinals (see \cite[Fact 2.5, Lemma 6]{MR2493990}), $L_{\Sigma}[r]=L_{\Sigma}^{r}$ by Lemma \ref{generic extension is subset of L stage} and Lemma \ref{L stage is subset of generic extension} and $L_{\Sigma}[r]$ is an increasing union of admissible sets by Lemma \ref{preservation of admissibility by randoms}. 
We choose some sufficiently random $s\in W_x$ over $L_\Sigma[r]$, in  particular $s$ is random over $L_{\Sigma+1}$. 
Since $L_{\lambda}\prec_{\Sigma_{1}}L_{\zeta}\prec_{\Sigma_{2}}L_{\Sigma}$, we have 
$$L_{\lambda}[r]\prec_{\Sigma_{1}}L_{\zeta}[r]\prec_{\Sigma_{2}}L_{\Sigma}[r]$$ 
by by Theorem \ref{Sigma2 reflection}, and we obtain the same 
elementary chain for $s$. 
Since $(\lambda^r,\zeta^r,\Sigma^r)$ and $(\lambda^s,\zeta^s,\Sigma^s)$ are lexically minimal and the values do not decrease in the extensions by $r$ and $s$, this implies $\lambda=\lambda^r=\lambda^s$, $\zeta=\zeta^r=\zeta^s$ and $\Sigma=\Sigma^r=\Sigma^s$. 

We can assume that $r$ is random over $L_\gamma$ and $s$ is random over $L_\gamma[r]$ for some $\gamma>\Sigma$ such that $L_\gamma$ satisfies a sufficiently strong theory to prove the forcing theorem and facts about random forcing, and such that generics and quasi-generics over $L_\gamma$ coincide (see \cite[Lemma 26.4]{Jech}). 
Since the $2$-step iteration of random forcing is equivalent to the side-by-side random forcing (see \cite[Lemma 3.2.8]{Bartoszynski-Judah}), 
$(r,s)$ is side-by-side random over $L_{\Sigma+1}$.\footnote{Alternatively, the proof of the product lemma or the $2$-step lemma \cite[Lemma 15.9, Theorem 16.2]{MR1940513} can easily be adapted to show directly that $(r,s)$ is side-by-side random over $L_{\Sigma+1}$.}

Since $x$ is writable relative to $r$ and relative to $s$, $x\in L_\lambda[r]\cap L_\lambda[s]=\lambda$ by Lemma \ref{intersection from mutually randoms}, therefore $x$ is writable. 
\end{proof}

As far as we know, the following class is the largest class between $\Pi^1_1$ and $\Sigma^1_2$ that has been studied. 
We write $x\leq_{n-\mathrm{hyp}}y$ if $x$ is computable from $y$ by a $\Sigma_n$-hypermachine introduced in \cite{FW}. 

\begin{thm} \label{variant for hypermachines} 
For all $n\geq 1$, a real $x$ is writable if and only if $\mu(\{y:x\leq_{n-\mathrm{hyp}}y\})>0$
\end{thm} 
\begin{proof} 
The proof is analogous to the proof of Theorem \ref{ITTMSacks} via the results of \cite{FW} and the version of Lemma \ref{Sigma2 reflection} for $\Sigma_n$-formulas instead of $\Sigma_2$-formulas. 
\end{proof}


\subsection{Recognizable reals from non-null sets} 

We will prove an analogous result as in the previous section, where computable reals are replaced with \emph{recognizable reals} from \cite{MR1771072}. 
This is an interesting and much stronger alternative notion to computability. 
The divergence between computability and recognizability is studied in \cite{MR1771072, recognizablesets}. 

A real is recognizable if its singleton is decidable. 
\emph{Lost melodies}, i.e. recognizable non-computable sets, do not appear in Turing computation, but already exists in the hyperarithmetic setting as $\Pi^1_1$ non-hyperarithmetic singletons. 

\begin{definition} \label{definition: recognizable} 
\begin{enumerate-(a)} 
\item 
A real $x$ is \emph{recognizable} if and only if there is an ITTM-program $P$ such that $P$ halts for every input $y$, with output $1$ if and only if $x=y$. 
\item 
A real $x$ is a  \emph{lost melody} if it is recognizable, but not writable. 
\end{enumerate-(a)} 
\end{definition} 


A simple example for a lost melody is the constructibly least code for a model of $ZFC+V{=}L$. 
It was demonstrated in \cite[Theorem 3.12]{Infinite-time-recognizability-from-random-oracles} that every real that is recognizable from all elements of a non-meager Borel set is itself recognizable. 
The new observation for the following proof is that one can avoid computing generics by working with the forcing relation. 
This also leads to a simpler proof in the non-meager case. 

\begin{thm} \label{recITTMsacks} 
Suppose that a real $x$ is recognizable from all elements of $A$ and $\mu(A)>0$. 
Then $x$ is recognizable.
\end{thm}
\begin{proof} 
We can assume that there is a single program $P$ which recognizes $x$ from all oracles in $A$, since the set of oracles which recognize $x$ for a fixed program is absolutely $\Delta^1_2$ and hence Lebesgue measurable (see \cite[Exercise 14.4]{MR2731169}).

\begin{claim} 
Let $D$ be the set of the conditions in $L_{\lambda^x}$ 
which decide whether $x$ is accepted or rejected by $P$ relative to the random real over $L_{\Sigma^x+1}$. 
Then $\mu(A\setminus \bigcup D)=0$.
\end{claim} 
\begin{proof} 
If the conclusion fails, then there is a random real $y$ over $L_{\Sigma^x+1}$ in $A\setminus \bigcup D$. 
Since $P^{x\oplus z}$ converges for any $z\in A$, $P^{x\oplus y}\downarrow=i$ for some $i$. 
Since $\lambda^{x\oplus y}=\lambda^x$ by Theorem \ref{Sigma2 reflection} and $L_{\lambda^x}[x\oplus y]= L_{\lambda^x}^{x\oplus y}$ by Lemma \ref{generic extension is subset of L stage} and Lemma \ref{L stage is subset of generic extension}, there is a name $\dot{C}$ in $L_{\lambda^x}$ and a condition $p$ in $L_{\lambda^x}$ with $y\in [p]$ which forces that $\dot{C}$ is a computation of $P$ with input $x\oplus y$ and output $i$. 
Then $p\in D$ and $y\in \bigcup D$, contradicting the assumption on $y$. 
\end{proof}


By the Lebesgue density theorem, there is an open interval with rational endpoints for which the relative measure of $A$ is $>1-\epsilon$ for some $\epsilon<\frac{1}{3}$. 
We can assume that this interval is equal to ${}^\omega 2$. 

The procedure $Q$ for recognizing $x$ works as follows. 
Suppose that $\dot{y}$ is a name for the random real over $L_{\Sigma+1}$. 
Given an oracle $z$, we enumerate $L_{\lambda^{z}}[z]$ via a universal ITTM. In parallel, we search for pairs $(p,\dot{C})$ in $L_{\lambda^{z}}[z]$ such that
$p$ is a condition and $\dot{C}$ is a name such that $p$ forces over $L_{\lambda^{z}}[z]$ that $\dot{C}$ is a computation of $P$ in the oracle $z\oplus \dot{y}$. 
that halts with output $0$ or $1$. Note that these are $\Delta_{0}$ statements and that the forcing relation for such statements is $\Delta_1$ by Lemma \ref{application of boolean values} and hence \ITTM-decidable.
We keep track of the conditions that force the corresponding computation to halt with output $0$ or with output $1$ on separate tapes. 
Moreover, we keep track of the measures $u_0$ and $u_1$ of the union of all conditions on the two tapes. 
Note that the measure of Borel sets can be computed in admissible sets by a $\Delta_1$-recursion and hence it is ITTM-computable. 
Since $\mu(A)>1-\epsilon$ and $\mu(A\setminus \bigcup D)=0$, eventually $u_0+u_1>1-\epsilon$. 
As soon as this happens, we output $1$ if $u_0>1-2\epsilon$ and $0$ otherwise. 
We claim that $Q^z$ outputs $1$ if and only if $z=x$. 

\begin{claim} 
$Q^{x}\downarrow=1$.
\end{claim} 
\begin{proof} 
The measure of a countable union of sets can
be approximated with arbitrary precision by unions of a finite number of sets. 
Since $\mu(A\setminus \bigcup D)=0$ and $\mu(A)>1-\epsilon$, $\mu(\bigcup D)>1-\epsilon$. 
there are disjoint conditions $p,q\in L_{\lambda^x}[x]$ with $\mu([p]\cup [q])>1-\epsilon$ such that $p$ forces $Q^{x\oplus \dot{y}}\downarrow=1$, 
and $q$ forces $P^{x\oplus \dot{y}}\downarrow=0$. 
Since $\mu(\bigcup D)>1-\epsilon$, $\mu([q])\leq\epsilon$ and hence $\mu([p])>1-2\epsilon$. 
Eventually, such a condition $p$ will be found and hence the procedure halts with output $1$. 
\end{proof} 

\begin{claim} 
$Q^{z}\downarrow=0$ if $z\neq x$. 
\end{claim} 
\begin{proof} 
Suppose that the claim fails. 
Since $Q$ always halts, we have 
$Q^{z}\downarrow=1$.
Then there is a condition $p$ with $\mu([p])>1-2\epsilon$ which forces $P^{z\oplus \dot{y}}\downarrow=1$. 
Since $\mu(A)>1-\epsilon$ and $\epsilon<\frac{1}{3}$, $\mu(A\cap [p])>0$ and hence there is a random $y$ in $A\cap [p]$ over $L_{\lambda^{z}}[z]$. 
Since $y\in [p]$, $P^{z \varoplus y}=1$. 
Since $y\in A$ and $z\neq x$, $P^{z\varoplus y}=0$. 
\end{proof} 
This completes the proof of Theorem \ref{recITTMsacks}. 
\end{proof} 

The results in Section \ref{section random forcing} also imply analogues of Theorem \ref{ITTMSacks} and Theorem \ref{recITTMsacks} for other notions of computation and recognizability, for instance the infinite time register machines \cite{MR2592054} and a weaker variant \cite{MR2592054}. We explore this in further work. 

\section{Random reals} 
\label{section random reals} 

We introduce natural randomness notions associated with infinite time Turing machines and show that they have various desirable properties. 

This is the motivation for the previous results, which we will apply here. 
The results resemble the hyperarithmetic setting, although some proofs are different. 
Theorem \ref{characterization of decidable random} shows a difference to the hyperarithmetic case.

\subsection{ITTM-random reals} \label{subsection ITTM-random reals} 
The following is a natural analogue to $\Pi^1_1$-random. 

\begin{defini} 
A real $x$ is \emph{\ITTM-random} if it is not an element of any \ITTM-semidecidable null set. 
The definition relativizes to reals. 
\end{defini}

We first note that there is a universal test. 
This follows from the following lemma, as in \cite[Theorem 5.2]{MR2340241}. 
\begin{lemma} \label{assign null covers} 
We can effectively assign to each \ITTM-semidecidable set $S$ an \ITTM-semi-decidable set $\hat{S}$ with 
$\mu(\hat{S})=0$, and $\hat{S}=S$ if $\lambda(S)=0$. 
\end{lemma} 
\begin{proof} 
Suppose that $S$ is an \ITTM-semi-decidable set, given by a program $P$. 
We define $S_\alpha$ as the set of $z$ such that $P(z)$ halts before $\alpha$. 
Note that if $M$ is admissible and contains a code for $\alpha$, then there is a Borel code for $S_\alpha$ in $M$ and hence $\mu(S_\alpha)$ can be calculated in $M$. In particular, $\mu(S_\alpha)$ is \ITTM-writable from any code for $\alpha$. 
Moreover, $\alpha$ is \ITTM-writable in $z$ since $\alpha<\lambda^z$. 
Hence there is a code for $\alpha$ in $L_{\lambda^z}$. 
Let $\hat{S}$  be the set of all $z$ such that there exists some $\alpha<\lambda^z$ with $z\in S_\alpha$ and $\mu(S_\alpha)=0$. 
Morover, let $\hat{S}_\alpha$ denote the set of $z$ with $z\in S_\alpha$ and $\mu(S_\alpha)=0$. 
Since the set of $z$ with $\lambda^z=\lambda$ is co-null by Theorem \ref{ITTMSacks}, $\hat{S}$ is the union of a null set and the sets $\hat{S}_\alpha$ for all $\alpha<\lambda$. 
\end{proof} 
The universal test is the union of all sets $\hat{S}$, where $S$ ranges over the \ITTM-semidecidable sets. 
The following notion is analogous to $\Pi^1_1$-random. 

The following is a variant of van Lambalgen's theorem for ITTMs. 
We say that reals $x$ and $y$ are \emph{mutually random}, in any given notion of randomness, if their join $x\oplus y$ is random. 

\begin{lemma} \label{hardITTMvanLambalgen} 
A real $x$ is \ITTM-random and a real $y$ is \ITTM-random relative to $x$ if and only if $x$ and $y$ are mutually \ITTM-random. 
\end{lemma} 
\begin{proof} 
Suppose that $x$ is \ITTM-random and $y$ is \ITTM-random relative to $x$. 
Moreover, suppose that $x$ and $y$ are not mutual \ITTM-randoms. 
Then there is an \ITTM-semidecidable set $A$ given by a program $P$ such that $x\oplus y\in A$. 
Let $A_u=\{v\mid u \oplus v\in A\}$ denote the section of $A$ at $u$. 
Let 
$$A_{>q} :=\{u\mid \mu(A_u)>q\}$$ 
for $q\in \mathbb{Q}$. 
Note that $u\in A_{>q}$ if and only if some condition in $L_{\Sigma^u}$ with measure $r>q$ in $\mathbb{Q}$ forces that $P(\check{u},\dot{v})$ halts, where $\dot{v}$ is a name for the random real over $L_{\Sigma^u}$, by Lemma \ref{application of boolean values}. 
This is a $\Sigma_1$-statement in $L_{\Sigma^u}$ and therefore in $L_{\lambda^u}$. 
Then the set $A_{>q}$ is semidecidable by Lemma \ref{characterization of ITTM-semidecidable}, uniformly in $q\in \mathbb{Q}$. 
Since $\mu(A)=0$, $\mu(A_{>0})=0$. 
Since $x$ is \ITTM-random, $x\notin A_{>0}$ and hence $\mu(A_x)=0$. 
Note that $A_x$ is semidecidable in $x$. 
Since $y$ is \ITTM-random relative to $x$, this implies $y\notin A_x$, contradicting the assumption that $x\oplus y\in A$. 

Now suppose that $x$ and $y$ are mutually \ITTM-random. 
To show that $x$ is \ITTM-random, suppose that $A$ is a semidecidable null set with $x\in A$. 
Then $A\oplus {}^\omega 2$ is a semidecidable null set containing $x\oplus y$, contradicting the assumption that $x$ and $y$ are mutually \ITTM-random. 
To show that $y$ is \ITTM-random relative to $x$, suppose that $y$ is an element of a semidecidable null set $A$ relative to $x$. 
Since the construction of $\hat{S}$ in Lemma \ref{assign null covers} is effective, there is a semidecidable null subset $B$ of ${}^\omega 2\times {}^\omega 2$ with $A=B_x$ (in fact, all sections of $B$ are null). 
Then $x\oplus y\in A$, contradicting the assumption that $x$ and $y$ are mutual \ITTM-randoms. 
\end{proof}

The following result is analogous to the statement that a real $x$ is $\Pi^1_1$-random can be characterized by $\Delta^1_1$-randomness and $\omega_1^x=\om$ (see \cite[Theorem 9.3.9]{MR2548883}). 

\begin{thm} \label{characterization of ITTM-random} 
A real $x$ is \ITTM-random if and only if it is random over $L_\Sigma$ and $\Sigma^x=\Sigma$. 
Moreover, this implies $\lambda^x=\lambda$. 
\end{thm} 
\begin{proof} 
First suppose that  $x$ is \ITTM-random. 
We first claim that $x$ is random over $L_\Sigma$. 
Since every real in $L_\Sigma$ is accidentally writable, 
we can enumerate all Borel codes in $L_\Sigma$ for sets $A$ with $\mu(A)=0$ and test whether $x$ is an element of $A$. 
Therefore the set of reals which are not random over $L_\Sigma$ is an ITTM-semidecidable set with measure $0$, and hence $x$ is random over $L_\Sigma$. 
We now claim that $\Sigma^x=\Sigma$. 
Since $\Sigma^y=\Sigma$ holds for all sufficiently random reals by Lemma \ref{Sigma2 reflection}, the set $A$ of reals $y$ with $\Sigma^y>\Sigma$ has measure $0$. 
Since the existence of 
$\Sigma$ is a $\Sigma_1$-statement over $L_{\Sigma^y}$, the set $A$ is semidecidable. 
Since $x$ is \ITTM-random, $x\notin A$ and hence $\Sigma^x=\Sigma$. 

Second, suppose that $x$ is random over $L_\Sigma$ and $\Sigma^x=\Sigma$. 
Suppose that $A$ is a semi-decidable null set containing $x$ given by a program $P$. 
Then $P(x)$ halts before $\lambda^x<\Sigma^x=\Sigma$ and hence some condition $p$ forces over $L_{\Sigma}$ that $P(x)$ halts, by Lemma \ref{application of boolean values}. 
Then $\mu(A)>0$, contradicting the assumption that $A$ is null. 

To show that $\lambda^x=\lambda$, note that $L_\lambda[x]\prec_{\Sigma_1} L_\Sigma[x]=L_{\Sigma^x}[x]$ by Lemma \ref{Sigma2 reflection}. 
Since $\lambda^x$ is minimal with this property, $\lambda^x\leq \lambda$. 
\end{proof} 

This shows that the level of randomness in the assumption of Lemma \ref{Sigma2 reflection}  can be improved to \ITTM-random for $\alpha=\zeta$, $\beta=\Sigma$. 

Surprisingly, we do not know if $\zeta^x=\zeta$ for \ITTM-randoms $x$. 
This does not follow from the proof of Lemma \ref{Sigma2 reflection} , since the set $\bar{A}$ defined in the beginning of the proof is not \ITTM-semidecidable, but this would be needed for the proof of Claim \ref{Sigma2 statement reflects} in the proof of Lemma \ref{Sigma2 reflection}. 

We obtain the following variant of Theorem \ref{ITTMSacks}. 

\begin{thm} \label{computable from mutual ITTM-randoms} 
If $x$ is computable from both $y$ and $z$ and $y$ is \ITTM-random in $z$, then $x$ is computable. 
In particular, this holds if $y$ and $z$ are mutual \ITTM-randoms. 
\end{thm} 
\begin{proof} 
Suppose that $P(y)=Q(z)=x$. 
Then $A=\{u\mid P(u)=Q(z)\}$ is semidecidable in $z$. 
If $\mu(A)>0$, then $x$ is computable from all element of a set of positive measure and hence $x$ is computable by Theorem \ref{ITTMSacks}. 
Suppose that $\mu(A)=0$. 
Then $y\notin A$, since $y$ is \ITTM-random in $z$, contradicting the assumption that $y\in A$. 
\end{proof}

\subsection{A decidable variant} \label{subsection ITTM decidable random} 

Martin-L\"of suggested to study $\Delta^1_1$-random reals.
The following variant of \ITTM-random is an analogue to $\Delta^1_1$-random. 

\begin{defini} 
A real is \emph{\ITTM-decidable random} if it is not an element of any decidable null set. 
\end{defini} 
We now give a  characterization of this notion. 
We call a real co-\ITTM-random if it avoids the complement of every semidecidable set of measure $1$. 
The following result is analogous to the equivalence of $\Delta^1_1$-random and $\Sigma^1_1$-random \cite[Exercise 14.2.1]{Yu}. 

\begin{thm} \label{characterization of decidable random}
The following properties are equivalent. 
\begin{enumerate-(a)} 
\item 
$x$ is co-\ITTM-random. 
\item 
$x$ is \ITTM-decidable random. 
\item 
$x$ is random over $L_\lambda$. 
\end{enumerate-(a)} 
\end{thm} 
\begin{proof} 
The first implication is clear. 

For the second implication, note that since every Borel set with a Borel code in $L_\lambda$ is \ITTM-decidable, every \ITTM-decidable random real $x$ is random over $L_\lambda$. 

For the remaining implication, suppose that $x$ is random over $L_\lambda$ and $P$ is a program that decides the complement of a null set $A$ with $x\in A$. 
Suppose that $\dot{x}$ is the canonical name for the random real (note that this name is equal for randoms over arbitrary admissible sets). 
Relative to the set of random reals $y$ over $L_{\Sigma+1}$, $A$ is definable over $L_\Sigma$, since $\Sigma^y=\Sigma$ by Theorem \ref{Sigma2 reflection}. 
Hence $y\notin A$ and $P(y)$ halts before $\lambda^y=\lambda$ for any such real. 
Therefore in $L_\Sigma$, there is some $\gamma$ (namely $\lambda$) such that the Boolean value of the statement that $P(\dot{x})$ halts strictly before $\gamma$ is equal to $1$. 
The existence of such an ordinal $\gamma$ is a $\Sigma_1$-statement, hence there is such an ordinal $\bar{\gamma}<\lambda$ such that the statement holds in $L_\lambda$ for $\bar{\gamma}$, by $\Sigma_1$-reflection. 
Let $A$ denote the Boolean value of the statement that $P(\dot{x})$ halts before $\bar{\gamma}$. 
Then $A$ is a Borel set with a Borel code in $L_\lambda$ and $\mu(A)=1$. 
Therefore $x\in A$ and $P(x)$ halts before $\lambda$, contradicting the assumption that $x\in A$. 
\end{proof} 

Hence the distance between the analogues to $\Delta^1_1$-random and $\Pi^1_1$-random is larger than for the original notions. 

\begin{lemma} 
There is no universal \ITTM-decidable random test. 
\end{lemma} 
\begin{proof} 
Suppose that $A$ is a universal \ITTM-decidable random test. 
In particular, the complement of $A$ is \ITTM-semidecidable. 
By the characterization of \ITTM-semidecidable reals in Lemma \ref{characterization of ITTM-semidecidable} and \cite[Seyfferth-Schlicht, Corollary 8]{MR3068304}, \ITTM-semidecidable uniformization holds.\footnote{The proof is a variant of the proof of $\Pi^1_1$-uniformization.} 
Therefore, every semidecidable set, in particular the complement of $A$, has a recognizable element. 
This contradicts the assumption that $A$ is a universal test. 
\end{proof} 

We call a program $P$ \emph{deciding} if $P(x)$ halts for every input $x$. 
The following is a version of van Lambalgen's theorem for \ITTM-decidable. 

\begin{lemma} \label{Lambalgen for decidable random} 
A real $x$ is \ITTM-decidable random and a real $y$ is \ITTM-decidable random relative to $x$ if and only if $x\oplus y$ is \ITTM-decidable random. 
\end{lemma} 
\begin{proof} 
Suppose that $x\oplus y$ is \ITTM-decidable random.  
The forward direction is a slight modification of the proof of von Lambalgen's theorem for \ITTM s in Lemma \ref{hardITTMvanLambalgen}, so we omit it. 
In the other direction, the only missing piece is the following claim. 

\begin{claim} 
Suppose that $A$ is a decidable set given and $A_x=\{y\mid x\oplus y \in A\}$ is null. 
Then there is a decidable set $B$ such that $A_x=B_x$ and all sections of $B$ are null. 
\end{claim} 
\begin{proof} 
It was shown in the proof of Lemma \ref{hardITTMvanLambalgen} that the set 
$$A_{>q}=\{u\mid \mu(A_u)>q\}$$ is semidecidable for all rationals $q$, uniformly in $q$, since the statement $u\in A_{>q}$ is $\Sigma_1$ over $L_{\Sigma^u}$. Since $L_{\lambda^u}\prec_{\Sigma_1}L_{\Sigma^u}$, the statement is $\Sigma_1$ over $L_{\lambda^u}$. 
Let 
$$A_{\geq q}=\{u\mid \mu(A_u)\geq q\},$$
Then the statement $u\in A_{\geq q}$ is equivalent to $u\in A_{>r}$ for unboundedly many rationals $r<q$. 
Since $\lambda^u$ is $u$-admissible, this is a $\Sigma_1$-statement in $u$ over $L_{\lambda^u}$. Hence $A_{\geq q}$ is semidecidable, uniformly in $q$. 

Therefore, if $A$ is decidable, then $A_{>q}$ and $A_{\geq q}$ are semidecidable, uniformly in $q$. 
Using the fact that  $A_{=0}=\{u\mid \mu(A_u)=0\}$ is decidable, it is easy to define a decidable set $B$ as in the claim. 
\end{proof} 
This completes the proof of Lemma \ref{Lambalgen for decidable random}.  
\end{proof} 

Lemma \ref{characterization of decidable random} and \ref{Lambalgen for decidable random} immediately imply that $x$ and $y$ are mutually random over $L_\lambda$ if and only if $x$ is random over $L_\lambda$ and $y$ is random over $L_{\lambda^x}$. 

The following variant of Lemma \ref{computable from mutual ITTM-randoms} for reals computable from two mutually randoms can be shown for the following stronger reduction. 
A \emph{safe \ITTM-reduction} of a real $x$ to a real $y$ is a deciding \ITTM\ (i.e. $P$ halts on every input) with $P(x)=y$.
We call reals $x$ and $y$ \emph{mutually \ITTM-decidable random} if $x\oplus y$ is \ITTM-decidable random. 

\begin{lemma} \label{Sacks for decidable random} 
If $x$ is safely \ITTM-reducible both to $y$ and $z$, and $y$ and $z$ are mutually \ITTM-decidable random, then $x$ is \ITTM-computable.
\end{lemma} 
\begin{proof} 
Suppose that $P$ is a safe reduction of $x$ to $y$ and $Q$ is a safe reduction of $x$ to $z$. 
Since $P$ is a safe reduction, the set $A=\{u \mid P(u) =Q(z)\}$ is \ITTM-decidable relative to $z$.
As $P(y)=x=Q(z)$, $y\in A$. 
Since $y$ is \ITTM-decidable relative to $z$, $A$ is not null. 
Then $P$ computes $x$ from all elements of a non-null Lebesgue measurable set, and hence $x$ is computable by \ref{ITTMSacks}. 
\end{proof} 

Lemma \ref{Lambalgen for decidable random} can be interpreted as the statement that $x$ and $y$ are mutually random (i.e. $x\oplus y$ is random) over $L_\lambda$ if and only if $x$ is random over $L_\lambda$ and $y$ is random over $L_{\lambda^x}$, by the relativized version of Lemma \ref{characterization of decidable random}. 

Intuitively, a random sequence should not be able to compute any non-computable sequence with special properties, such as recognizable sequences. 
The following result confirms this. 

\begin{lemma}{\label{compfromrecog}}
Any recognizable real $x$ that is computable from an \ITTM-random real $y$ is already computable.
\end{lemma}
\begin{proof} 
Suppose that $P$ recognizes $x$ and $Q(y)=x$. Then the set 
$$A=\{z\mid P^{Q(z)}=1\}$$  
is semi-decidable and contains $y$, where $Q(z)$ is the output of the computation $Q$ with input $z$. 
Note that $x$ is computable from every element of $A$ via $Q$. 
If $A$ is not null, then $x$ is computable by Theorem \ref{ITTMSacks}. 
If $A$ is null, this contradicts the assumption that $y$ is \ITTM-random and thus avoids $A$. 
\end{proof}

Hence there are real numbers that are not computable from any \ITTM-random real, and therefore there is no analogue for \ITTM-randoms to the Ku\v cera-G\'acs theorem (see \cite[Theorem 8.3.2]{MR2732288}).

\begin{remark} 
All previous results and proofs 
work relativized to reals and for arbitrary continuous measures instead of the Lebesgue measure. 
\end{remark}

\subsection{Comparison with a Martin-L\"of type variant } \label{subsection ML} 

We finally consider a Martin-L\"of variant of ITTM-randomness. The importance of this notion lies in its characterization via initial segment complexity. This variant is strictly between ITTM-random and $\Pi^1_1$-random. 

We first describe analogues of the theorems of van Lambalgen and Levin--Schnorr for $\mathrm{ITTM}_{\mathrm{ML}}$-random reals. 
Since these results are minor modifications of the results in \cite{MR2340241} and \cite{Continuous-higher-randomness}, we refer the reader to \cite[Section 3]{MR2340241} and \cite[Section 1.1, Section 3]{Continuous-higher-randomness} for discussions and proofs, and will only point out the differences to our setting. 

Towards proving van Lambalgen's theorem for $\mathrm{ITTM}_{\mathrm{ML}}$-random reals, we define a continuous relativization as in \cite[Section 1.1]{Continuous-higher-randomness}. 
If $\Psi\subseteq {}^\omega2\times {}^\omega2$ and $x\in {}^\omega2$, let 
$$\Psi^{(x)}=\{n\mid (\sigma,m)\in \Psi\text{ for some }\sigma\preceq x\}.$$ 
A subset $A$ of ${}^\omega 2$ is called $\mathrm{ITTM}_{\mathrm{ML}}^{(x)}$ if $A=\Psi^{(x)}$ for some $\mathrm{ITTM}$-semidecidable set $\Psi$. 

\begin{lemma} 
A real $x\oplus y$ is \ML-random if and only if $x$ is \ML-random and $y$ is $\mathrm{ITTM}^{(x)}_{\mathrm{ML}}$-random. 
\end{lemma} 

The difference to the proof in \cite[Section 3]{Continuous-higher-randomness} is that $\om$ is replaced with $\lambda$ and the projectum function on $\om$ is replaced with a projectum function on $\lambda$, i.e. an injective function  $p\colon \lambda\rightarrow \omega$ such that its graph is $\Sigma_1$-definable over $L_\lambda$. 
For instance, consider the function $p$ which maps an ordinal $\alpha<\lambda$ to the least 
program that writes a code for $\alpha$. 

The proof of the Levin-Schnorr theorem in \cite[Theorem 3.9]{MR2340241} easily adapts 
to our setting as follows, by replacing $\om$ with $\lambda$ and $\Pi^1_1$-random with \ML-random. 
We will also call an \ITTM\  simply a \emph{machine}. 

\begin{lemma} 
There is an effective list $\langle M_d\mid d\in \omega\setminus \{0\}\rangle$ of all prefix-free ITTMs. 
\end{lemma} 
\begin{proof} 
We can effectively replace each machine  $P$ by a prefix-free machine  $\hat{P}$, by simulating $P$ on all inputs with increasing length. 
\end{proof} 

Given such a list, we obtain a universal prefix-free machine $U$ by defining $U(0^{d-1}1\sigma)=M_d(\sigma)$. 
We identify $U$ with an semidecidable subset of $2^{<\omega}\times 2^{<\omega}$. The ITTM-version of Solomonoff-Kolmogorov complexity is defined as 
$$K(x)=K_U(x)= \min \{|\sigma|\mid U(\sigma)=x\}.$$ 

\begin{definition} 
Suppose that $D$ is a prefix-free machine. 
The probability that $D$ outputs a string $x$ is 
$P_D(x)=\lambda(\{\sigma\mid D(\sigma)=x\})$. 
\end{definition} 

By the definition of $K$, $2^{-K(x)}\leq P_U(x)$. 
As in \cite[Theorem 3.4]{MR2340241} we have the following result. 

\begin{thm} (Coding theorem) 
For each prefix-free machine $D$, there is a $\Sigma_1$-definable function $g\colon \lambda\rightarrow \lambda$ over $L_\lambda$ and a constant $c$ such that 
$$ \forall x\ 2^c 2^{-K(x)}\geq P_D(x).$$ 
\end{thm} 

This implies the following analogue to the Levin-Schnorr theorem to characterize randomness via incompressibility, as in \cite[Theorem 3.9]{MR2340241}. 

\begin{thm} 
The following properties are equivalent for infinite strings $x$. 
\begin{enumerate-(a)} 
\item 
$x$ is \ML-random. 
\item 
$\exists b\ \forall n\ K(x\upharpoonright n)>n-b$. 
\end{enumerate-(a)} 
\end{thm} 

The difference to the proof of \cite[Theorem 3.9]{MR2340241} is that $\om$ is replaced with $\lambda$ and the coding theorem for the \ITTM-variant of $K$ is used. 

We now compare the introduced randomness notions with $\Pi^1_1$-randomness. 
There is an \ITTM-writable $\Pi^1_1$-random real, for example, let $x$ be the $<_L$-least real that is random over $L_{\om+1}$. 
Since $L_\lambda$ is admissible and $\om$ is countable in $L_\lambda$, $x\in L_\lambda$. 
Then $x$ is $\Pi^1_1$-random by Lemma \ref{preservation of admissibility by randoms} and Lemma \ref{classical characterization of Pi11-randoms}, 
and all reals in $L_\lambda$ are ITTM-computable. 

For the next result, recall that a real $r\in \RR$ is called \emph{left-$\Pi^1_1$} if the set $\{q\in \mathbb{Q}\mid q\leq r\}$ is $\Pi^1_1$. 
The following is a folklore result and we give a short proof for the benefit of the reader. 

\begin{lemma} \label{measure of Pi11 sets} (Tanaka, see \cite[Section 2.2 page 15]{MR0369072}) 
The measure of $\Pi^1_1$ sets is uniformly left-$\Pi^1_1$. 
\end{lemma} 
\begin{proof} 
Using the Gandy-Spector theorem and Sacks' theorem (see Lemma \ref{preservation of admissibility by randoms}) that the set of reals $x$ 
with $\omega_1^x=\om$ has full measure, we can associate to a given $\Pi^1_1$ set a sequence of length $\om$ of hyperarithmetic subsets, such that their union approximates the set up to measure $0$. 
This shows that the measure is left-$\Pi^1_1$. 
Moreover, in the proof of the Gandy-Spector theorem (see \cite[Theorem 5.5]{Hjorth-Vienna-notes-on-descriptive-set-theory}) for a $\Pi^1_1$ set ${}^\omega 2\setminus p[T]$, the $\Sigma_1$-formula states that $T_x$ is well-founded, and hence the parameter in the formula is uniformly computable from $T$, and the assignment is uniform. 
\end{proof} 

\begin{lemma} 
Every \ITTM-random is \ML-random and every \ML-random is $\Pi^1_1$-random. 
\end{lemma} 
\begin{proof} 
The first implication is obvious. 
For the second implication, suppose that $A=p[T]$ is a $\Sigma^1_1$. 
Using Lemma \ref{measure of Pi11 sets}, we inductively build finitely splitting subtrees $S_n$ of $T$ with $\mu([T]\setminus [S_n])\leq 2^{-n}$, uniformly in $n$. 
This sequence can be written by an \ITTM. 
\end{proof} 
\section{Questions} 

We conclude with several open questions. 
Surprisingly, the proof of Theorem \ref{Sigma2 reflection} does not answer the following question. 

\begin{question} 
Is $\zeta^x=\zeta$ for every \ITTM-random $\zeta$? 
\end{question} 

%

Moreover, we have left open various questions about the connections between randomness notions and their properties. 
The following question asks if a property of $\mathrm{ML}$-random and $\Delta^1_1$-random (see \cite[Theorem 14.1.10]{Yu}) holds in this setting. 

\begin{question} 
Is \ML-random strictly stronger than random over $L_\lambda$? 
\end{question} 

The fact that \ML-random is strictly stronger than $\Pi^1_1$-random suggests an analogue for $\Sigma_n$-hypermachines. 

\begin{question} 
Is every $\mathrm{ML}$-random with respect to $\Sigma_{n+1}$-hypermachines already semidecidable random with respect to $\Sigma_n$-hypermachines? 
\end{question}

Since the complexity of the set of $\Pi^1_1$-randoms is $\bf{\Pi}^0_3$ \cite[Corollary 27]{MR3181446} and this is optimal (see \cite[Theorem 28]{MR3181446} and \cite{MR2860186}), this suggests the following question. 

\begin{question} 
What is the complexity of the set of \ITTM-random reals? 
\end{question} 

The set $\mathrm{NCR}$ is defined as the set of reals that are not random with respect to any continuous measure. 
It is known that this set has different properties in the hyperarithmetic setting \cite{MR3436362} and for randomness over the constructible universe $L$ \cite{YuZhu}. 

\begin{question} 
Is there concrete description of the set $\mathrm{NCR}$, defined with respect to \ITTM-randomness? 
\end{question} 

Moreover, it is open whether Theorem \ref{computable from mutual ITTM-randoms} fails for \ML-randomness. More precisely, we can ask for an analogue to the counterexample or $\mathrm{ML}$-randomness (see \cite[Section 5.3]{MR2548883}). 

\begin{question} 
Let $\Omega_0$ and $\Omega_1$ denote the halves of the \ITTM-version of Chaintin's $\Omega$ (i.e. the halting probability for a universal prefix-free machine). 
Is some non-computable real computable from both $\Omega_0$ and $\Omega_1$? 
\end{question}

\bibliographystyle{alpha} 
\bibliography{references}

\end{document}